\documentclass[11pt,a4paper]{amsart}
\pdfoutput=1

\pdfpageattr{/Group <</S /Transparency /I true /CS /DeviceRGB>>}

\usepackage[T1]{fontenc}
\usepackage[american]{babel}
\usepackage[autostyle=true]{csquotes}
\usepackage{geometry}
\usepackage{amsfonts,amsmath,amssymb,amsthm,mathrsfs}
\usepackage{bbm}
\usepackage{tikz}
\usetikzlibrary{calc,shapes,decorations.markings,decorations.text,hobby}
\usepackage{ifthen}
\usepackage{pgfplots}
\usepackage{caption}
\usepackage{capt-of}
\usepackage{subcaption}
\usepackage{xspace}
\usepackage{booktabs,tabularx}
\usepackage{dcolumn}
\usepackage{marvosym}
\usepackage{microtype}
\usepackage{enumerate}
\usepackage{mathtools}
\usepackage{multirow}
\usepackage{hyperref}

\DeclareMathOperator\realtravel{lt}
\DeclareMathOperator\freetravel{ft}
\DeclareMathOperator\bias{B}
\DeclareMathOperator\capacity{cap}
\DeclareMathOperator\power{pow}

\newcommand\combinatorialTypes{\sigma} % number of combinatorial types of shortest path trees

\microtypesetup{tracking,kerning,spacing}
\microtypecontext{spacing=nonfrench}

\hypersetup{colorlinks,linkcolor=blue ,citecolor=blue, urlcolor=blue}

\newcommand\polymake{{\tt polymake}\xspace}

\newcommand\neato{{\tt neato}\xspace}
\newcommand\graphviz{{\tt Graphviz}\xspace}

\newcommand\NN{{\mathbb N}}

\newcommand\RR{{\mathbb R}}
\newcommand\ZZ{{\mathbb Z}}

\newcommand\TT{{\mathbb T}}
\newcommand\cS{\mathcal{S}}

\newcommand\cT{\mathcal{T}}

\newcommand\SetOf[2]{\left\{\left.#1\vphantom{#2}\ \right|\ #2\vphantom{#1}\right\}}
\newcommand\smallSetOf[2]{\{{#1}\,|\,{#2}\}}

\newcommand\graphOfShortestPathTrees{\mathcal G}

\renewcommand{\epsilon}{\varepsilon}
\renewcommand{\phi}{\varphi}

\theoremstyle{plain}
    \newtheorem{theorem}{Theorem}
    \newtheorem{corollary}[theorem]{Corollary}
    \newtheorem{lemma}[theorem]{Lemma}
    \newtheorem{proposition}[theorem]{Proposition}
\theoremstyle{definition}
    \newtheorem{remark}[theorem]{Remark}
    \newtheorem{example}[theorem]{Example}

    \newtheorem{observation}[theorem]{Observation}

    \newtheorem{Algorithm}{Algorithm}

\title[Parametric shortest-path algorithms via tropical geometry]{Parametric shortest-path algorithms \\ via tropical geometry}

\author{Michael Joswig \and Benjamin Schr\"oter}

\address[Michael Joswig]{
  Technische Universit\"at Berlin, Chair of Discrete Mathematics/Geometry,
  and MPI MiS, Leipzig, Germany
}
\email{joswig@math.tu-berlin.de}

\address[Benjamin Schr\"oter]{
  Department of Mathematics,
  KTH Royal Institute of Technology,
  Stockholm, Sweden
}
\email{schrot@kth.se}

\thanks{%
This research has been carried out in the framework of Matheon supported by Einstein Foundation Berlin.
Additional support by Institut Mittag-Leffler within the program \enquote{Tropical Geometry, Amoebas and Polytopes} is gratefully acknowledged.
M.~Joswig has further been supported by the Deutsche Forschungsgemeinschaft (DFG, German Research Foundation) under Germany's Excellence Strategy - The Berlin Mathematics Research Center MATH$^+$ (EXC-2046/1, project ID 390685689);  \enquote{Symbolic Tools in Mathematics and their Application} (TRR 195, project-ID 286237555); \enquote{Facets of Complexity} (GRK 2434).
Furthermore, we are very grateful for the support by several colleagues.
Ewgenij Gawrilow's work on the \polymake project is crucial; here he implemented the parameterized Dijkstra algorithm.
Max Klimm was helpful in discussing the application to traffic.
Georg Loho directed our attention to the article \cite{Gallo+Grigoriadis+Tarjan:1989}.
Two anonymous referees provided helpful suggestions to improve the exposition.
Thanks to everybody.
}

\subjclass[2020]{%
  14T90   % Applications of tropical geometry
  (05C12, % (1991-now) Distance in graphs
  68R10, %  Graph theory (including graph drawing) in computer science
  90B20, % Traffic problems in operations research}
  90C31)} % (1980-now) Sensitivity, stability, parametric optimization
\keywords{parametrized shortest paths; Dijkstra's algorithm; traffic networks; tropical geometry}

\begin{document}

\begin{abstract}
  We study parameterized versions of classical algorithms for computing shortest-path trees.
  This is most easily expressed in terms of tropical geometry.
  Applications include shortest paths in traffic networks with variable link travel times.
\end{abstract}
\maketitle

\section{Introduction}
\noindent
One of the most basic classes of algorithmic problems in combinatorial optimization is the computation of shortest paths for all pairs of nodes in a directed graph.
The reader may consult the monograph of Schrijver \cite{Schrijver03:CO_A} for a comprehensive survey.
Here we study parameterized versions where some of the arc weights are unspecified.
It turns out that standard tools such as the Floyd--Warshall algorithm \cite[\S8.4]{Schrijver03:CO_A} or Dijkstra's algorithm \cite[\S7.2]{Schrijver03:CO_A} admit interesting generalizations.
While it is known that the shortest path problem is connected to max-plus linear algebra and tropical geometry (see, e.g., \cite[Chap.~4]{Butkovic10}, \cite[\S5.2]{Tropical+Book}
and their references), this paper investigates how the geometric underpinnings can be exploited algorithmically.

Dijkstra's algorithm and its siblings are among the core tools used, e.g., in devices which help a car driver to navigate a road network.
These efficient methods allow for solving the corresponding shortest path problems almost instantly, even on cheap hardware, and even for fairly large networks.
Methods from robust optimization have been used to take some uncertainty about the link travel times into account, see, e.g., \cite{YuYang:1998}, \cite{MR3907225} and the references there.
Yet the situation for the network provider is quite different from the perspective of the network user.
One reason is that the provider's goal does not necessarily agree with the one of the user:
While the individual driver might be interested in short travel times, the traffic authorities of a metropolitan city might want to, e.g., minimize the total amount of pollution.
More importantly, the traffic authorities seek to achieve a system optimum, whereas the driver cares for an individual objective; cf.\ \cite{Moehring:2005}.
Typically, in relevant cases it is next to impossible to even describe a system optimum.
Here we propose methods from tropical geometry to gain new insight.
For instance, this may be applied to assess the impact of local changes to a network a priori, provided that the number of simultaneous modifications is not too high.

Tropical geometry is a thriving field, which combines algebraic and polyhedral geometry with optimization; cf.~\cite{Tropical+Book}.
The benefit is mutual: geometry gains algorithmic methods through optimization, and topics in theoretical computer science and optimization may be analyzed more deeply through geometric insight.
Examples for the latter include interpreting the decision problem MEAN-PAYOFF in terms of tropical linear programs \cite{AkianGaubertGutermann12}, ordinary linear programs which exhibit a non-strongly polynomial behavior for interior point methods \cite{ABGJ:2018} or applications to multicriteria optimization~\cite{JoswigLoho:1707.09305}.
Examples for the former abound, too; to name just one: the perfect matching problem, classically solved by the Hungarian method \cite[\S17.2]{Schrijver03:CO_A}, computes tropical determinants \cite[Prop.~1.2.5]{Tropical+Book}.
It is worth noting that our parameterized version of computing shortest paths is of particular interest geometrically.
Namely, this is related to enumerating a class of convex polytopes known as \emph{polytropes} or \emph{alcoved polytopes} and has been studied, e.g., by Tran~\cite{Tran:2017}.

Our setup is the following.
Let $\Gamma$ be a directed graph with $n$ nodes and $m$ arcs.
Throughout we will assume that $\Gamma$ has no parallel arcs.
Additionally, each arc will be equipped with a weight.
Then the graph together with the weight function can be encoded as an $n{\times}n$-matrix where the coefficient at position $(u,v)$ is the weight on the arc from $u$ to $v$.
Necessarily we have $m\leq n^2$, with equality if and only if $\Gamma$ is a complete directed graph with $n$ loops.
Since we will be interested in shortest path problems we consider smaller weights as better, and this suggests to use $\infty$ to signal the absence of an arc.
The resulting matrix is a \emph{weighted adjacency matrix} of $\Gamma$.
This leaves the question: what are the weights?

In the context of the shortest path problem a very general answer is the following.
Let $(G,+)$ be a totally ordered abelian group such that $\infty$ is not an element of~$G$.
Then $G\cup\{\infty\}$, equipped with \enquote{$\min$} as the addition and \enquote{$+$} as the multiplication, is a semiring; this is the \emph{$(\min,+)$-semiring} associated with $G$.
Here $\infty$ is neutral with respect to the addition and absorbing with respect to the multiplication.
Via the usual rules for the addition and the multiplication of matrices this entails a semiring structure on the set $(G\cup\{\infty\})^{n\times n}$ of $n{\times}n$-matrices with coefficients in $G\cup\{\infty\}$.
The classical shortest path problem occurs when $G$ is the additive group of the real numbers.
We denote the extension $\RR\cup\{\infty\}$ of this group by $\TT$.
However, it is interesting and useful to go one step beyond by only requiring that $G$ is a commutative semigroup equipped with a partial ordering, which is not necessarily total.
Then, in general, for a pair of nodes there are competing shortest paths whose total weights are incomparable.
We will see that basic algorithmic ideas for solving shortest path problems still remain valid, with minor adjustments.
The case when $G$ is the additive semigroup of tropical polynomials with real coefficients in a fixed number of indeterminates is of particular interest to us.

To look at parameterized versions of shortest path problems is not a new idea.
A first paper which explores connections to polyhedral geometry is Fredman \cite{Fredman:1976}.
Another important precursor of our approach is a paper by Gallo, Grigoriadis and Tarjan \cite{Gallo+Grigoriadis+Tarjan:1989} on a parametric version of the celebrated push--relabel method for computing maximum flows by Goldberg and Tarjan \cite{Goldberg+Tarjan:1988}.
Moreover, shortest path computations have been considered in the context of robust optimization; cf.\ \cite{MR2546839} for a general reference.
For instance, Yu and Yang observed that in a digraph equipped with interval weights, for given nodes $s$ and $t$, it is NP-complete to decide whether there is a shortest $s$--$t$ path whose total weight stays below a certain threshold \cite[Theorem~1]{YuYang:1998}.
Other modern concepts in this area include online techniques (e.g., see \cite{AlonEtAl:2006}) as well as robustness combined with randomization (e.g., see \cite{MatuschkeSkutellaSoto:2015}) and dynamic algorithms (e.g., see \cite{Bernstein:2016}).
The $s$--$t$ shortest path problems addressed in the above models are relevant, e.g., for a single driver who wants to navigate her car through a road network with uncertain link travel times.
Here instead we are considering the all-pairs shortest path problem, which amounts to taking the perspective of the provider of the network; cf.\ Section~\ref{sec:computations} below for computational experiments on real-world data.
These show that our method is also practically useful, provided the output size is not too large.

Our paper is organized as follows.
Section~\ref{sec:floyd-warshall} starts out with a brief sketch on how to generalize the classical algorithm of Floyd and Warshall to the scenario with parameterized arc weights.
Standard results on tropical hypersurfaces are invoked to reveal basic structural insight into the shortest-path problem.
The algorithmic core of this paper, explained in Section~\ref{sec:dijkstra}, is a procedure for enumerating all parameterized shortest-path trees to a fixed target node.
This can be seen as a parameterized analog to Dijkstra's algorithm.
We demonstrate that this algorithm is feasible in practice for few variable arc weights.
Our computational results are summarized in Sections~\ref{sec:computations} and~\ref{sec:concluding}.

\section{Parameterizing the Floyd--Warshall algorithm} 
\label{sec:floyd-warshall}
\noindent
A standard method for computing all shortest paths between any pair of nodes in a directed graph is the Floyd--Warshall algorithm.
This is well-known to have a straightforward interpretation in tropical arithmetic as follows.
We will briefly sketch the method and refer to \cite[\S8.4]{Schrijver03:CO_A} or \cite[Section 5.9]{Ahoetal:1975} for details.

Let $\Gamma$ be a directed graph on $n$ nodes with weighted adjacency matrix $D=(d_{uv})_{u,v}\in \TT^{n\times n}$.
A naive algorithm for obtaining all-pairs shortest paths is to compute the $(n{-}1)$st tropical power $(I\oplus D)^{\odot(n-1)}$ of the matrix $I\oplus D$ where $I=D^{\odot 0}$ is the tropical identity matrix, with coefficients $0$ on the diagonal and $\infty$ otherwise. 
Here we used \enquote{$\oplus$} for the tropical matrix addition, which is defined as the coefficientwise minimum, and \enquote{$\odot$} for the tropical matrix multiplication, i.e., the analog of classical matrix multiplication where \enquote{$\min$} and \enquote{$+$} replace the addition and the multiplication, respectively. In particular, $I\oplus D$ is the weighted adjacency matrix of $\Gamma$ with zeros along the diagonal if the graph has no negative loops. 
For computing shortest paths the $(n{-}1)$st power is enough since any shortest path, if it exists, takes at most $n-1$ arcs.
Each of the $n-2$ multiplications takes $O(n^3)$ time, resulting in a total cost of $O(n^4)$.
We will not discuss clever strategies for multiplying these matrices as we will beat this naive matrix multiplication approach via \eqref{eq:floyd-warshall} below.
Unless there are negative cycles the coefficient of $D^{\odot(n-1)}$ at position $(u,v)$ is the length of a shortest path from node $u$ to~$v$ using exactly $n-1$ arcs, and the coefficient of $(I\oplus D)^{\odot(n-1)}$ at position $(u,v)$ is the length of a shortest path from node $u$ to~$v$.
Moreover, a negative cycle exists if and only if a coefficient on the diagonal of $(I\oplus D)^{\odot n}$ or $D$ is negative.
Formally, the solution to the all-pairs shortest path problem can be written as
\begin{equation}\label{eq:kleene}
  D^* \ = \ I \oplus D \oplus D^{\odot 2} \oplus \cdots \oplus D^{\odot(n-1)} \oplus \cdots \enspace ,
\end{equation}  
which converges to $(I\oplus D)^{\odot (n-1)}= I\oplus \dots \oplus D^{\odot(n-1)}$ if and only if there is no negative cycle.
The matrix $D^*$ is called the \emph{Kleene star} of $D$; cf.\ Butkovi\v{c} \cite[\S1.6.2.1]{Butkovic10}.

Floyd and Warshall's algorithm reduces the complexity of computing $D^*$ to $O(n^3)$ via dynamic programming.
The key ingredient is the weight of a shortest path from $u$ to $v$ with all intermediate nodes restricted to the set $\{1,2,\dots,r\}$, which is
\begin{equation}\label{eq:floyd-warshall}
  d_{uv}^{(r)} \ = \
  \begin{cases}
    0 & \text{if } r=0 \text{ and } u= v\\
    d_{uv} & \text{if } r=0 \text{ and } u\neq v\\
    \min\left( d_{uv}^{(r-1)},\, d_{ur}^{(r-1)}+d_{rv}^{(r-1)}\right) & \text{if } r\geq 1 \enspace .
  \end{cases}
\end{equation}
That is, in the nontrivial step of the computation we check if going through the new node~$r$ gives an advantage.

We set $D^{(r)}=\big(d_{uv}^{(r)}\big)_{u,v}$.
By applying the formula \eqref{eq:floyd-warshall} recursively, the Floyd--Warshall algorithm computes $D^{(n)}$ in $O(n^3)$ time.
The trick is that, with $D^{(r-1)}$ known explicitly, the computation of a single coefficient $d_{uv}^{(r)}$ requires only constant time.
Note that this method is also suitable for detecting negative cycles by checking the diagonal of the result.
A negative cycle exists if and only if some diagonal coefficient of $D^{(n)}$ is negative.
Otherwise $D^{(n)}=(I \oplus D)^{\odot(n-1)}=D^*$.
In general, $D^{(r)}$ is distinct from any tropical power~$(I\oplus D)^{\odot k}$ whenever $1\leq r< n-1$.

\begin{remark}
  For computing all-pairs shortest path in a dense graph with arbitrary weights there is no algorithm known that beats the $O(n^3)$ complexity bound; see \cite[\S8.6]{Schrijver03:CO_A}.
  Yet Floyd and Warshall's algorithm was improved by Fredman in \cite{Fredman:1976} when the edge weights are restricted to be nonnegative.
  Fredman's bound is based on the reduction of computing the Kleene star $D^*$ to tropical matrix multiplication; see \cite[Theorem~5.7 and Corollary~2]{Ahoetal:1975}.
  Moreover, he subdivides the matrix multiplication into the multiplication of smaller block matrices of sizes $n\times\sqrt{n}$ and $\sqrt{n}\times n$, respectively.
  In combination with a clever search method this leads to a bound of $O(n^{5/2})$ comparisons.
  Further, this approach leads to an algorithm of complexity $O(n^3 \log\log(n)/\log(n)^2)$ by Han and Takaoka \cite{Han+Takaoka:2016}; see \cite[\S7.5]{Schrijver03:CO_A} for a general overview of all-shortest paths with nonnegative weights and \cite{Han+Takaoka:2016} for more recent developments.
\end{remark}

\begin{figure}[t]
  \centering
  \begin{minipage}{.47\textwidth}\centering
    \begin{tikzpicture}[scale = 0.7,
                    color = {black}]

\tikzset{->-/.style={decoration={
  markings,
  mark=at position #1 with {\arrow{>}}},postaction={decorate}}}
  \tikzstyle{edge} = [->-=0.7,> = stealth, thin, line cap=round, line join=round];

  \tikzstyle{linestyle} = [ultra thick, line cap=round, line join=round];
  \tikzstyle{linestyle2} = [thin, gray, line cap=round, line join=round];

  \draw[very thin,color=gray] (-3.6,-3.6) grid (3.6,3.6);
  \draw[->, >=latex, color=black!50, thick] (-4.2,0) -- (4.6,0) node[right, color=black] {$x$};
  \draw[->, >=latex, color=black!50, thick] (0,-4.2) -- (0,4.6) node[right, color=black] {$y$};
  
  % min x+1, y, 2
  \coordinate (l0) at (1,2);
  \coordinate (l1) at (-3.6,-2.6);
  \coordinate (l2) at (1,3.6);
  \coordinate (l3) at (3.6,2);
  
  \draw [linestyle] (l0) -- (l1);
  \draw [linestyle] (l2) -- (l0) -- (l3);
  \node [circle,draw,fill=white] at (l0) {};

  \node[fill=white] at (2,3) {$2$};
  \node[fill=white] at (-2,2) {$x+1$};
  \node[fill=white] at (1,-1) {$y$};
  
\end{tikzpicture}

% Local Variables: 
% mode: latex
% mode: TeX-PDF
% TeX-master: "../main"
% mode: reftex
% mode: font-lock
% buffer-file-coding-system:utf-8-unix
% End: 
  \end{minipage}
  \hfill
  \begin{minipage}{.47\textwidth}\centering
    % $Q=new Hypersurface<Min>(POLYNOMIAL=>toTropicalPolynomial("min(4+2*x0,x0+x1+3,x0+x2+9/2,x1+x2+2,2x1+4,2x2+6)"));
% > print rows_labeled($Q->VERTICES);
% 0:0 0 1 0
% 1:0 0 0 1
% 2:0 0 -1 -1
% 3:1 0 -1 1
% 4:1 0 1 1
% 5:1 0 5/2 -1/2
% 6:1 0 5/2 -3/2
% > print $Q->REGIONS;
% {0 2 6}
% {2 3 4 5 6}
% {1 2 3}
% {0 5 6}
% {0 1 4 5}
% {1 3 4}

\begin{tikzpicture}[scale = 0.7,
                    color = {black}]
                    
\tikzset{->-/.style={decoration={
  markings,
  mark=at position #1 with {\arrow{>}}},postaction={decorate}}}
  \tikzstyle{edge} = [->-=0.7,> = stealth, thin, line cap=round, line join=round];

  \tikzstyle{linestyle} = [ultra thick, line cap=round, line join=round];
  \tikzstyle{linestyle2} = [thin, gray, line cap=round, line join=round];

  \draw[very thin,color=gray] (-3.6,-3.6) grid (3.6,3.6);
  \draw[->, >=latex, color=black!50, thick] (-4.2,0) -- (4.6,0) node[right, color=black] {$x$};
  \draw[->, >=latex, color=black!50, thick] (0,-4.2) -- (0,4.6) node[right, color=black] {$y$};
  
  \coordinate (q3) at (-1,1);
  \coordinate (q4) at (1,1);
  \coordinate (q5) at (2.5,-0.5);
  \coordinate (q6) at (2.5,-1.5);
  
  \coordinate (q3up) at (-1,3.6);
  \coordinate (q4up) at (1,3.6);
  \coordinate (q5right) at (3.6,-0.5);
  \coordinate (q6right) at (3.6,-1.5);

  \coordinate (q3ll) at (-3.6,-1.6);
  \coordinate (q6ll) at (0.4,-3.6);
  
  \node[fill=white] at (2.5,2.5) {$4$};
  \node[fill=white] at (0,2.5) {$3+x$};
  \node[fill=white] at (-2.5,2.5) {$4+2x$};
  \node[fill=white,right] at (3,-1) {$\tfrac{9}{2}+y$};
  \node[fill=white] at (2.5,-3) {$6+2y$};
  \node[fill=white] at (-0.5,-1.5) {$2+x+y$};

  \draw [linestyle] (q3) -- (q4) -- (q5) -- (q6);
  \draw [linestyle] (q3ll) -- (q3) -- (q3up);
  \draw [linestyle] (q4) -- (q4up);
  \draw [linestyle] (q5) -- (q5right);
  \draw [linestyle] (q6ll) -- (q6) -- (q6right);
  
  \node [circle,draw,fill=white] at (q3) {};
  \node [circle,draw,fill=white] at (q4) {};
  \node [circle,draw,fill=white] at (q5) {};
  \node [circle,draw,fill=white] at (q6) {};
  
\end{tikzpicture}

% Local Variables: 
% mode: latex
% mode: TeX-PDF
% TeX-master: "../main"
% mode: reftex
% mode: font-lock
% buffer-file-coding-system:utf-8-unix
% End: 
  \end{minipage}

  \caption{Two generic tropical plane curves.
    Each region is marked with the term at which the minimum is attained.
    Left: tropical line defined by $\min (2, 1+x, y)$.  Right: tropical quadric defined by $\min (4, 3+x, 4+2x, 2+x+y, 6+2y, \frac{9}{2}+y)$.}
  \label{fig:tropical}
\end{figure}
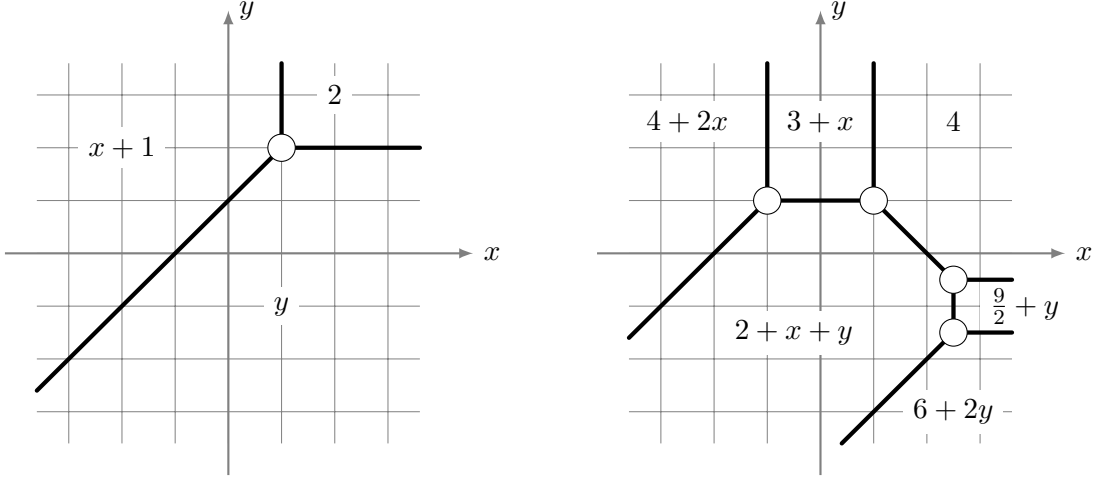

Our first observation is that the same ideas can be applied in the presence of variable arc weights.
To this end we consider a weighted adjacency matrix where each coefficient is a multivariate polynomial whose coefficients lie in the $(\min,+)$-semiring $\TT=\RR\cup\{\infty\}$.
These polynomials again form a semiring, and thus, via the usual addition and multiplication, the set of $n{\times}n$-matrices with coefficients in $\TT[x_1,\dots,x_k]$ is a semiring, too.

Formally, a $k$-variate tropical polynomial, $f\in\TT[x_1,\dots,x_k]$, with finite \emph{support set} $S\subset\ZZ^k$, is a map which assigns each exponent $a\in S$ a coefficient $\gamma_a\in\RR$.
This gives rise to the \emph{evaluation function} $f(t_1,\dots,t_k) = \min\smallSetOf{\gamma_a + a_1 t_1 + \dots + a_k t_k}{a\in S}$ which sends $t\in\RR^k$ to a real number.
That function is piecewise linear, continuous and concave; cf.\ \cite[\S1.1]{Tropical+Book}.
For each $a\in S$ the set $\smallSetOf{t\in\RR^k}{f(t)=\gamma_a + a_1 t_1 + \dots + a_k t_k}$ is a convex polyhedron.
That set is a \emph{region} of $f$ if that polyhedron is of maximal dimension $k$.
The regions are the domains of linearity of $f$, and they form a polyhedral subdivision of $\RR^k$; cf.\ Figure~\ref{fig:tropical}.
The \emph{tropical hypersurface} $\cT(f)$ is the locus where  at least two terms of $f$ are minimal; i.e., $\cT(f)$ \enquote{lies between} pairs of regions.

Now, if $D$ is a matrix with coefficients in $\TT[x_1,\dots,x_k]$, then each coefficient of $D$ defines a tropical hypersurface and a polyhedral subdivision of~$\RR^k$.
The maximal cells of the common refinement of the polyhedral subdivision taken over all coefficients of $D$ are the \emph{regions} of $D$.
Each region of $D$ is the intersection of the regions of the tropical hypersurfaces corresponding to some set of coefficients of $D$.

\begin{observation}\label{obs:initial}
  The solution to the all-pairs shortest paths problem of a directed graph with $n$ nodes and weighted adjacency matrix $D\in\TT[x_1,\dots,x_k]$ is a polyhedral decomposition of $\RR^k$ induced by up to $n^2$ tropical hypersurfaces corresponding to the nonconstant coefficients of $D^{\odot (n-1)}$.
  On each polyhedral cell the lengths of all shortest paths are linear functions in the $k$ parameters.
\end{observation}

By multiplying the nonconstant tropical polynomials of a matrix $D\in\TT[x_1,\dots,x_k]^{n\times n}$ we obtain one tropical polynomial which yields the tropical hypersurface \emph{induced} by $D$, and this is denoted by $\cT(D)$.
The regions of $\cT(D)$ are precisely the regions of $D$.

% data/initial.graph
% with extension from git@git.polymake.org:extensions/polytropes

% $A=network2matrix($initial);
% local_var_names<Polynomial<TropicalNumber<Min>, Int>>(qw(x y));

\begin{example}\label{exmp:initial} % data/initial.graph
  The weighted adjacency matrix
  % polytope > print tex_matrix($A);
  \begin{equation}\label{eq:initial}
    D \ = \ \begin{pmatrix}
      0 & \infty & \infty & 1 \\
      1 & 0 & \infty & \infty \\
      y & 1 & 0 & \infty \\
      \infty & x & 1 & 0
    \end{pmatrix} \enspace ,
  \end{equation}
  whose coefficients lie in the semiring $\TT[x,y]$ of bivariate tropical polynomials, defines a directed graph, $\Gamma$, on four nodes.
  The third tropical power of $D=I\oplus D$ reads
  % print tex_matrix($A*$A*$A);
  \[
    D^{\odot 3} \ = \ \small\begin{pmatrix}
      \min( 2+x, 2+y, 0 ) & \min( 1+x, 3 ) & 2 & 1 \\
      1 & \min( 2+x, 0 ) & 3 & 2 \\
      \min( y, 2 ) & \min( 1+x+y, 1 ) & \min( 2+y, 0 ) & \min( 1+y, 3 ) \\
      \min( 1+x, 1+y, 3 ) & \min( x, 2 ) & 1 & \min( 2+x, 2+y, 0 )
    \end{pmatrix} \enspace .
  \]
  Ten among the 16 coefficients $d_{uv}^{\odot 3}$ are nonconstant.
  Three of the corresponding tropical polynomials are linear and generic, i.e., they have three terms: $d_{1,1}^{\odot 3} = d_{4,4}^{\odot 3} = \min( 2+x, 2+y, 0 )$ and $d_{4,1}^{\odot 3} = \min( 1+x, 1+y, 3 )$.
  Six coefficients are linear but degenerate.
  Among these we have the equalities $d_{1,2}^{\odot 3} = \min( 1+x, 3 ) = 1 + d_{4,2}^{\odot 3}$ and $d_{3,4}^{\odot 3} = \min( 1+y, 3) = 1 + d_{3,1}^{\odot 3}$.
  Only the coefficient $d_{3,2}^{\odot 3} = \min(1+x+y,1)$ is nonlinear; it is a (degenerate) tropical polynomial of degree two.
  In Figure~\ref{fig:initial} the resulting two nondegenerate and four degenerate tropical lines are marked by circles at their apices.
  Those lie at infinity in the degenerate cases.
  The tropical quadric degenerates to an ordinary line, which is marked by two squares.
  We obtain an arrangement of $2+4+1=7$ tropical hypersurfaces.
  Their product is the tropical hypersurface $\cT(D^{\odot 3})$ induced by $D^{\odot 3}$, up to multiplicities, and it has $15$ regions.
  
  Now we want to extract information about shortest paths from this geometric data.
  The diagonal of $D^{\odot 3}$ reveals that there are no negative cycles unless $x<-2$ or $y<-2$.
  All coefficients are finite, and thus $\Gamma$ is strongly connected.
  The \emph{feasible domain} is the set
  \[
    \SetOf{(x,y)\in\RR^2}{x\geq -2,\, y\geq -2} \enspace ,
  \]
  where shortest paths between any two nodes exist.
  It is subdivided into seven regions, four bounded and three unbounded ones.
  Eight of the $15$ regions of $D^{\odot 3}$ are infeasible.
\end{example}

\begin{figure}[th]\centering
  \begin{tikzpicture}[scale = 1.3,
                    color = {black}]

\tikzset{->-/.style={decoration={
  markings,
  mark=at position #1 with {\arrow{>}}},postaction={decorate}}}
  \tikzstyle{edge} = [->-=0.7,> = stealth, thin, line cap=round, line join=round];

  \tikzstyle{linestyle} = [ultra thick, line cap=round, line join=round];
  \tikzstyle{linestyle2} = [thin, gray, line cap=round, line join=round];

   \shade [inner color=gray!70, outer color=white] (-4.6,4) -- (-2,3.6) -- (-2,-2) -- (3.6,-2) -- (4,-4.6) -- (-4.8,-4.8) -- cycle;

  \draw[very thin,color=gray] (-3.6,-3.6) grid (3.6,3.6);
   \draw[->, >=latex, color=black!50, thick] (-4.2,0) -- (4.6,0) node[right, color=black] {$x$};
   \draw[->, >=latex, color=black!50, thick] (0,-4.2) -- (0,4.6) node[right, color=black] {$y$};

   \coordinate (mXX) at ($(-2,-2)$);
   \coordinate (m22) at ($( 2, 2)$);
   \coordinate (mX2) at ($(-2, 2)$);
   \coordinate (m2X) at ($( 2,-2)$);
   \coordinate (m23) at ($( 2, 3.6)$);
   \coordinate (m32) at ($(3.6, 2)$);
   \coordinate (mX3) at ($(-2, 3.6)$);
   \coordinate (m3X) at ($(3.6,-2)$);

	\node [fill=white,inner sep=0pt] at (3,3){\tiny$\arraycolsep=0.4\arraycolsep\ensuremath{
		\begin{pmatrix}
			0 & 3 & 2 & 1\\
			1 & 0 & 3 & 2\\
			2 & 1 & 0 & 3\\
			3 & 2 & 1 & 0
		\end{pmatrix}$}
	};

	\node [fill=white,inner sep=0pt] at (3,0){\tiny$\arraycolsep=0.3\arraycolsep\ensuremath{
		\begin{pmatrix}
			0 & 3 & 2 & 1\\
			1 & 0 & 3 & 2\\
			y & 1 & 0 & 1\!+\!y\\
			1\!+\!y & 2 & 1 & 0
		\end{pmatrix}$}
	};

	\node [fill=white,inner sep=0pt] at (0,3){\tiny$\arraycolsep=0.3\arraycolsep\ensuremath{
		\begin{pmatrix}
			0 & 1\!+\!x & 2 & 1\\
			1 & 0 & 3 & 2\\
			2 & 1 & 0 & 3\\
			1\!+\!x & x & 1 & 0
		\end{pmatrix}$}
	};

	\node [fill=white,inner sep=0pt] at (0,1.5){\tiny$\arraycolsep=0.3\arraycolsep\ensuremath{
		\begin{pmatrix}
			0 & 1\!+\!x & 2 & 1\\
			1 & 0 & 3 & 2\\
			y & 1 & 0 & 1+y\\
			1+x & x & 1 & 0
		\end{pmatrix}$}
	};
	\node [fill=white,inner sep=0pt] at (0,-1.5){\tiny$\arraycolsep=0.3\arraycolsep\ensuremath{
		\begin{pmatrix}
			0 & 1\!+\!x & 2 & 1\\
			1 & 0 & 3 & 2\\
			y & 1\!+\!x\!+\!y & 0 & 1\!+\!y\\
			1\!+\!y & x & 1 & 0
		\end{pmatrix}$}
	};

	\node [fill=white,inner sep=0pt,rotate=90] at (-1.45,0){\tiny$\arraycolsep=0.3\arraycolsep\ensuremath{
		\begin{pmatrix}
			0 & 1\!+\!x & 2 & 1\\
			1 & 0 & 3 & 2\\
			y & 1\!+\!x\!+\!y & 0 & 1\!+\!y\\
			1\!+\!x & x & 1 & 0
		\end{pmatrix}$}
	};

	\node [fill=white,inner sep=0pt, rotate=270] at (1.4,0){\tiny$\arraycolsep=0.3\arraycolsep\ensuremath{
		\begin{pmatrix}
			0 & 1\!+\!x & 2 & 1\\
			1 & 0 & 3 & 2\\
			y & 1 & 0 & 1\!+\!y\\
			1\!+\!y & x & 1 & 0
		\end{pmatrix}$}
	};

   \draw[linestyle] (m22) -- (-4,-4);
   \draw[linestyle] (2,-4.5) -- (m23);
   \draw[linestyle] (-4.5,2) -- (m32);
   \draw[linestyle] (-2,-4.5) -- (mX3);
   \draw[linestyle] (-4.5,-2) -- (m3X);

   \draw[linestyle] (-4.5, 4.5) -- (4.5,-4.5);

   \foreach \i in {-3,...,3}{
    \pgfmathtruncatemacro{\x}{20-2*\i*\i-1.1*\i}%
    \node [circle,fill=gray!\x, inner sep = 2pt]at (\i,-3.6){\small $\i$};
    \node [circle,fill=gray!\x, inner sep = 2pt]at (-3.6,\i){\small $\i$};
    }
   \node [circle,draw,fill=white, inner sep = 2pt] at (mXX)  {$d^{\odot 3}_{1,1}$};% = $a^{\odot 3}_{4,4}$};
   \node [circle,draw,fill=white, inner sep = 2pt] at (-2  ,-4.5){$d^{\odot 3}_{2,2}$};% = a^{(\odot 3)}_{4,2}$};
   \node [circle,draw,fill=white, inner sep = 2pt] at ( 2  ,-4.5){$d^{\odot 3}_{1,2}$};
   \node [circle,draw,fill=white, inner sep = 2pt] at (-4.5, 2)  {$d^{\odot 3}_{3,1}$};% = a^{\odot 3}_{3,4}$};
   \node [circle,draw,fill=white, inner sep = 2pt] at (-4.5,-2)  {$d^{\odot 3}_{3,3}$};
   \node [circle,draw,fill=white, inner sep = 2pt] at (m22){$d^{\odot 3}_{4,1}$};

   \node [rectangle,draw,fill=white, inner sep = 3pt] at (4.5,-4.5){$d^{\odot 3}_{3,2}$};
   \node [rectangle,draw,fill=white, inner sep = 3pt] at (-4.5,4.5){$d^{\odot 3}_{3,2}$};

\end{tikzpicture}

% Local Variables: 
% mode: latex
% mode: TeX-PDF
% mode: reftex
% mode: font-lock
% TeX-master: "../main"
% buffer-file-coding-system:utf-8-unix
% End: 
  \caption{%
    Decomposition of $\RR^2$ into the $15$ regions of the Kleene star $D^{\odot 3}$ from Example~\ref{exmp:initial}.
    These are induced by an arrangement of six tropical lines and one tropical quadric.
    On each of the seven feasible regions the lengths of the shortest paths are linear functions in $x$ and $y$.
    The eight infeasible regions are shaded.
  }
  \label{fig:initial}
\end{figure}
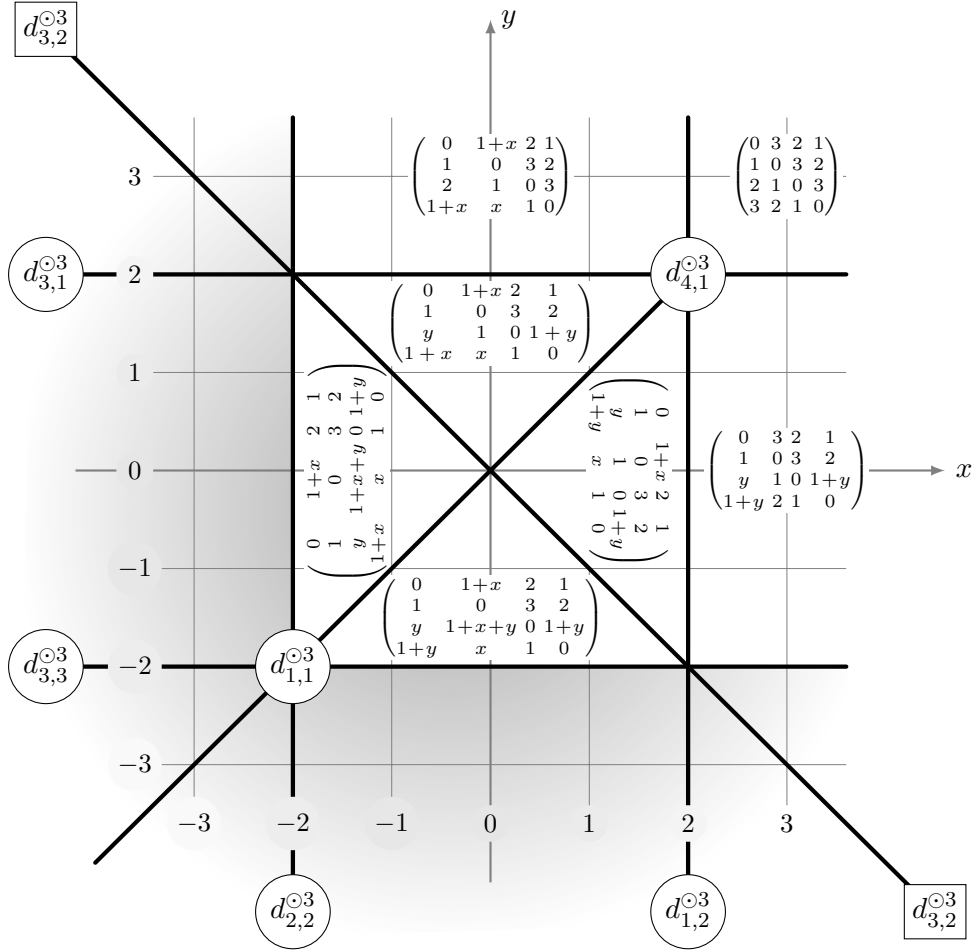

Comparing tropical polynomials $f,g\in\TT[x_1,\dots,x_k]$ as real functions we set
\begin{equation}\label{eq:partial-ordering}
  f \leq g \ :\!\iff \ f(z) \leq g(z) \quad \text{for all } z\in\TT^k \enspace .
\end{equation}
This defines a partial ordering.
It is easy to see that $f$ and $g$ are comparable if and only if $f=\infty$, $g=\infty$ or $f(z)=c+g(z)$ for some constant $c\in\RR$.

Consider a matrix $D\in\TT[x_1,\dots,x_k]^{n\times n}$.
We say that $D$ has \emph{separated variables} if each indeterminate occurs, with multiplicity one, in the weight of at most one arc.
In this way there is no dependence among the weights for distinct arcs.
Then each coefficient of $D$ involves a constant plus a sum of indeterminates. We may reduce the number of variables by substituting the sum of indeterminates by a single variable.
Thus we assume that the entries of $D$ are linear tropical polynomials.
It then follows that $k \leq m \leq n^2$.
This property is satisfied by the matrix \eqref{eq:initial} in Example~\ref{exmp:initial}.

\begin{theorem}\label{thm:floyd-warshall}
  Let $D\in\TT[x_1,\dots,x_k]^{n\times n}$ be the weighted adjacency matrix of a directed graph on $n$ nodes with separated variables.
  Then, between any pair of nodes, there are at most $2^k$ pairwise incomparable shortest paths.
  Moreover, the Kleene star $D^*$, which encodes all parameterized shortest paths, can be computed in $O(k \cdot 2^k \cdot n^{3})$ time, if it exists.
\end{theorem}

\begin{proof}
  We consider the case without negative cycles, i.e., cycles whose total weight is comparable to and strictly less than zero.
  Then there is at least one shortest path between any two nodes; for convenience here we take paths of weight $\infty$ into account.
  In each shortest path each arc occurs at most once unless the total weight of the path is $\infty$.
  By our assumption this means that the total weight is equal to $\lambda+x_{i_1}+\dots+x_{i_\ell}$ for some $\lambda\in\TT$ and $x_{i_1}+\dots+x_{i_\ell}$ is a multilinear tropical monomial, i.e., each indeterminate occurs with multiplicity zero or one.
  There are $2^k$ distinct multilinear monomials, and hence this bounds the number of incomparable shortest paths between any two nodes.

  To obtain our complexity result we use the Floyd--Warshall algorithm with the computation of the coefficients $d_{uv}^{\odot(r)}$ via \eqref{eq:floyd-warshall} as the key step.
  In our parameterized scenario each coefficient of $D^{\odot(r-1)}$ is a multilinear tropical polynomial.
  The tropical multiplication, i.e., ordinary sum, of two multilinear monomials takes linear time in the number of indeterminates, which is at most $k$.
  Each coefficient of $D^{(r-1)}$ has at most $2^k$ terms by our bound on the number of incomparable shortest paths.
  We infer that computing $d_{uv}^{(r)}$ takes not more than $O(k\cdot 2^k)$ time, and this yields our claim.
\end{proof}

If the coefficients of $D\in\TT[x_1,\dots,x_k]^{n\times n}$ are linear tropical polynomials, then each coefficient in $D^{\odot(n-1)}$ is a tropical polynomial of degree at most $n-1$, and thus the degree of the tropical hypersurface $\cT(D^{\odot (n-1)})$ does not exceed $n^2(n-1)$, which is of order $O(n^3)$.
This occurs, e.g., when $D$ has separated variables.
\begin{corollary}\label{cor:floyd-warshall}
  With the same conditions as in Theorem~\ref{thm:floyd-warshall}, and if $k$ is considered a fixed constant, all parameterized shortest paths can be computed in $O(n^{3})$ time, if they exist.
\end{corollary}

\begin{remark}
  The special case when \emph{all} the arc weights are variable is of particular interest to tropical geometry.
  This is equivalent  to enumerating a class of convex polytopes known as \emph{polytropes}; cf.\ Tran~\cite{Tran:2017}.
  The precise connection to our results is beyond the scope of this article.
\end{remark}

\section{A parameterized analog to Dijkstra's algorithm}
\label{sec:dijkstra}
\noindent
The Floyd-Warshall algorithm considered in the previous section is very useful to get a conceptual overview of the shortest-path problem. 
The Kleene star $D^*=(I\oplus D)^{\odot (n-1)}=D^{(n)}$ itself does not directly provide us with the information about all shortest paths for all choices of parameters simultaneously.
Instead this is only determined by the polyhedral decomposition of the parameter space $\RR^k$ into the regions of $D^*$, induced by the tropical hypersurface $\cT(D^*)$.
In this section we propose a method, inspired by Dijkstra's algorithm, to find the regions of $D^*$, given $D$.

Dijkstra’s algorithm is the main method for computing shortest paths used in applications; cf.~\cite[\S7.2]{Schrijver03:CO_A}.
It computes a shortest-path tree directed toward a fixed node.
In this setting it is common to assume that all weights are nonnegative, and this is what we will do here.

Let $\Gamma$ be a directed graph with $n$ nodes, without parallel arcs, and weighted adjacency matrix $D=(d_{uv})_{u,v}\in \TT[x_1,\dots,x_k]^{n\times n}$.
Working with nonnegative weights means that we consider the feasible region of the matrix $D$ within the positive orthant.
The nonnegativity assumption entails that a shortest path from any node to any other node is well defined or, equivalently, the Kleene star $D^*$ exists.
Since we do not assume that $\Gamma$ is strongly connected, we allow for \enquote{shortest paths} of infinite length.

Motivated by an application to traffic networks (cf.\ Section~\ref{sec:computations}) we choose the following setup.
Each arc $(u,v)$ in $\Gamma$ is equipped with a \emph{weight interval} $[\lambda_{uv},\mu_{uv}]$ subject to
\[
  0 \ \leq \ \lambda_{uv} \ \leq \ \mu_{uv} \ \leq \ \infty \enspace .
\]
If $\mu_{uv}=\infty$ then we abuse the notation $[\lambda_{uv},\mu_{uv}]$ for the ray $\SetOf{x\in\RR}{x\geq\lambda_{uv}}$.
Similarly, if $\lambda_{uv}=\mu_{uv}=\infty$, then $[\infty,\infty]$ is the empty set which signals the absence of an arc.
In Section~\ref{sec:computations} below, the weight interval will describe a range of possible travel times along a link in a traffic network.
We explicitly allow for the case $\lambda_{uv}=\mu_{uv}$, i.e., the arc $(u,v)$ may be equipped with a constant weight.
Assuming that there are precisely $k$ arcs with nonconstant weights, we can identify those arcs with the variables, for which we also use the notation $x_{uv}$.
Conversely, we also write \enquote{$\lambda(x_i)$} for the given lower bound on $x_i$ and \enquote{$\mu(x_i)$} for the given upper bound.
Setting the coefficients of $D$ to
\[
  d_{uv} \ = \
  \begin{cases}
    x_{uv} & \text{if } \lambda_{uv}<\mu_{uv} \\
    \lambda_{uv} & \text{otherwise} \enspace ,
  \end{cases}
\]
we arrive at the case of separated variables.
So Theorem~\ref{thm:floyd-warshall} applies, but here we restrict the feasible domain to the polyhedron $[\lambda(x_1),\mu(x_1)]\times\dots\times[\lambda(x_k),\mu(x_k)]$ in $\RR^k$.
That polyhedron is compact, if and only if all upper bounds are finite.

From now on we will compare two tropical polynomials $f,g\in\TT[x_1,\ldots,x_k]$ with respect to this polyhedron, i.e., we let $f\leq g$ if $f(z)\leq g(z)$ for all $z\in[\lambda(x_1),\mu(x_1)]\times\dots\times[\lambda(x_k),\mu(x_k)]$.
Whenever the tropical polynomials $f$ and $g$ are sums of the arc weights on paths with separated variables, then checking $f\leq g$ can be done with a single evaluation of each of the polynomials: that is, in this case $f\leq g$ if and only if $f(\mu)\leq g(\lambda)$ where $\lambda$ and $\mu$ are the minimum and the maximum, respectively, taken over all variable bounds involved.
If not further specified then $\lambda(x_i) = 0$ and $\mu(x_i)=\infty$.

A compact way to represent the set of shortest paths to a single target is a shortest-path tree.
A shortest-path tree is the result of Dijkstra's algorithm when all weights are constant.
Motivated by \cite[Theorem~7.3]{Tarjan:1983}, we extend the notion of shortest-path trees to the partial ordering $\leq$.
We call a spanning tree $T$ with all edges directed toward the target node $t$ a \emph{shortest-path tree} if, for every arc $(v,w)$, 
\begin{equation}\label{eq:shortest-path-tree}
	d_{vw}+p_w \ \not< \ p_v \enspace ,
\end{equation}
where $p_v$ is the length of the path from $v$ to $t$ in the tree $T$.
Often we will denote such a directed spanning tree as a pair $(T,p)$ in order to stress that all subsequent complexity bounds require the function $p$ to be given explicitly.
A direct consequence of \eqref{eq:shortest-path-tree} is:
\begin{observation}\label{obs:number_ineq}
  A given directed spanning tree $(T,p)$ has $n-1$ arcs out of the $m$ arcs of~$\Gamma$. Thus there are at most $m-n+1$ arcs $(u,v)$ such that $d_{vw}+p_w$ and $p_v$ are incomparable.
  In particular, it can be tested in $O(k\cdot m)$ time whether a directed spanning tree is a shortest-path tree.
\end{observation}

The solution to computing shortest-path trees toward the node $t$ in a directed graph with $n$ nodes and weighted adjacency matrix $D\in\TT[x_1,\dots,x_k]$ is a polyhedral decomposition $\cS$ of $\RR^k$ induced by up to $n-1$ tropical hypersurfaces corresponding to the nonconstant coefficients in the column labeled $t$ in the Kleene star $D^*$.
Note that all diagonal entries are zero as there are no negative cycles.
On each polyhedral cell the lengths of all shortest paths are linear functions in the $k$ parameters.
Each such cell is a union of cells of the subdivision induced by the tropical hypersurface~$\cT(D^*)$.

\begin{lemma} 
  Every shortest-path tree $(T,p)$ gives rise to a polyhedral cell in the decomposition~$\cS$.
  That cell is described by the inequalities
  \begin{equation}\label{eq:shortest-path-cell}
	  p_v \ \leq \ d_{vw}+p_w \quad \text{for all arcs } (v,w) \enspace. 
  \end{equation}
  Every region of $D$ arises in this way.
\end{lemma}
\begin{proof}
  The inequalities \eqref{eq:shortest-path-cell} are linear, and thus they define a polyhedron $P=P(T)$.
  For every nonnegative point $x\in\TT^k$ Dijkstra's algorithm produces a shortest-path tree, $T_x$, and we have $x\in P(T_x)$.
  As a consequence these polyhedra cover the feasible domain.

  The terms $p_v$ and $d_{vw}+p_w$ appear in the entry $d^{(n)}_{v,t}$ of the Kleene star $D^*$.
  Thus each cell of $\cS$ is either contained in $P$, or they are disjoint.
  On the other hand, if $x\in P(T)$ then every path in $T$ has to be a shortest path after substitution of the variables.
  In other words, if $q$ is a term of $d^{(n)}_{v,t}$ which is minimal for $x$ then $p_v$ and $q$ evaluate to the same value at $x$.
  This implies that $P$ is contained in a cell of $\cS$.
  We conclude that $P\in\cS$, and every region is of that form.
\end{proof}

Clearly, it is enough to take only those arcs into account for which $p_v$ is incomparable to $d_{vw}+p_w$.
The following example shows that a shortest-path tree $T$ may yield a lower dimensional cell or even the empty set.
\begin{example} \label{ex:infeasible}
  Consider the directed graph on four nodes shown in Figure~\ref{fig:infeasible} whose weights lie in the semiring $\TT[x]$ of univariate tropical polynomials.
  The Kleene star of its weighted adjacency matrix is
  \begin{equation}\label{eq:infeasible}
    D^* \ = \ \begin{pmatrix}
      0 & \infty & \infty & \infty \\
      x & 0 & \infty & \infty \\
      \min( 2+x, 5 ) & 2 & 0 & \infty \\
      \min( 3+x, 4 ) & 3 & \infty & 0
    \end{pmatrix} \enspace .
  \end{equation}
  The first column of $D^*$ yields four shortest-path trees with node $a$ as the target.
  The four corresponding systems of inequalities read
  \[
    x \leq 1 \ , \quad 1 \leq x \leq 3 \ , \quad 3 \leq x \ , \quad x \leq 1 \text{ and } 3\leq x \enspace ;
  \]
  where the final system is infeasible.
  That is, there are only three regions.
\end{example}
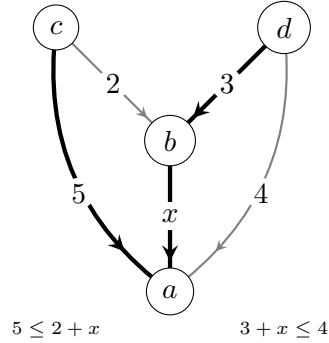
\begin{figure}[t]
  \centering
  \begin{tikzpicture}[scale = 1.0,
                    color = {black}]

%   \draw[white] (-2,-1) -- (2,1);

  \tikzset{->-/.style={decoration={
  markings,
  mark=at position #1 with {\arrow{>}}},postaction={decorate}}}
  \tikzstyle{linestyle} = [->-=0.8,> = stealth, ultra thick, line cap=round, line join=round];
  \tikzstyle{linestyle2} = [gray, ->-=0.8,> = stealth, thick, line cap=round, line join=round];

  \coordinate (m1) at ($( 0,-2)$);
  \coordinate (m2) at ($( 0, 0)$);
  \coordinate (m3) at ($(-1.5, 1.5)$);
  \coordinate (m4) at ($( 1.5, 1.5)$);

 \draw[linestyle] (m2) to (m1);
 \node [circle,fill=white,inner sep=1pt] at ($0.5*(m2)+0.5*(m1)$){$x$};
 \draw[linestyle2] (m3) to (m2);
 \node [circle,fill=white,inner sep=1pt] at ($0.5*(m3)+0.5*(m2)$){$2$};
 \draw[linestyle] (m4) to (m2);
 \node [circle,fill=white,inner sep=1pt] at ($0.5*(m4)+0.5*(m2)$){$3$};
 \draw[linestyle] (m3) to[bend right] (m1);
 \node [circle,fill=white,inner sep=1pt] at ($0.5*(m3)+0.5*(m1)-(.45,.45)$){$5$};
 \draw[linestyle2] (m4) to[bend left] (m1);
 \node [circle,fill=white,inner sep=1pt] at ($0.5*(m4)+0.5*(m1)+(.45,-.45)$){$4$};

   \node [circle,draw,fill=white] at (m1){$a$};
   \node [circle,draw,fill=white] at (m2){$b$};
   \node [circle,draw,fill=white] at (m3){$c$};
   \node [circle,draw,fill=white] at (m4){$d$};

\node [fill=white,inner sep=1pt] at ($(-1.5,-2.5)$){\tiny $5\leq 2+x$};
\node [fill=white,inner sep=1pt] at ($( 1.5,-2.5)$){\tiny $3+x\leq 4$};

\end{tikzpicture}
  \caption{The shortest-path tree in the directed graph of Example~\ref{ex:infeasible} that does not correspond to a feasible region.}
  \label{fig:infeasible}
\end{figure}

\begin{remark}\label{rem:infeasible}
  In Example~\ref{ex:infeasible} the single arc of variable weight is the only outgoing arc of the node $b$ and ends at the target node $a$.
  More general, suppose $G$ is a simple directed graph with $n$ nodes and $k$ arcs of variable weight such that they all end at the target node $t$ and these arcs are the only outgoing arcs of the corresponding nodes.	
  Every node (other than the target node) with an outgoing arc of constant weight has up to $k+1$ incomparable paths to the target.
  This leads to at most $(k+1)^{n-k-1}$ shortest-path trees.
  However, it can be shown that the number of regions of the graph $G$ does not exceed $\tbinom{n-1}{k}$.
  As the $n-k-1$ nodes with at most $k+1$ incomparable shortest path to $t$ correspond to $n-k-1$ tropical linear polynomials in $k$ variables.
  For which it is known that their common refinement has at most $\tbinom{n-1}{k}$ maximal cells. 
  As illustrated in Figure~\ref{fig:infeasible}, there are shortest-path trees that do not correspond to feasible regions.
\end{remark}

\begin{remark}
  Finding the dimension of a polyhedral cell given in terms of linear inequalities can be reduced to solving linear programs; cf.\ \cite[Theorem 6.5.5]{GLS}.
\end{remark}

Our aim is it to enumerate all shortest-path trees, and hence, via solving linear programs, all maximal dimensional polyhedral regions.
We will discuss our choices and other options at the end of this section.
We consider the graph $\graphOfShortestPathTrees=\graphOfShortestPathTrees(D)$ whose nodes are all shortest-path trees, and which has an edge between two nodes if the corresponding trees share $n-2$ common edges, i.e., there is exactly one node $u$ with two outgoing edges, and the two paths from $u$ to the target $t$ are incomparable. 

\begin{remark}\label{rem:dualgraph} 
  The graph $\graphOfShortestPathTrees(D)$ contains the dual graph of the polyhedral subdivision $\cS$ as a connected subgraph.
\end{remark}

A graph traversal enumerates all nodes in the connected component of some first node.
This is the core of our approach, which employs the following two procedures.

\begin{Algorithm}[Find an initial shortest-path tree]\label{alg:startnode}
  Set each unknown $x_i$ to its minimal value $\lambda_i$.
  Run Dijkstra's algorithm to obtain a shortest-path tree, with fixed arc weights, for the target node $t$.
  Let $T$ be this shortest-path tree, equipped with the original weights.
  For each node $u$ this yields a parameterized distance $p^T_u\in\TT[x_1,\ldots,x_k]$ from $u$ to $t$ in~$T$.
\end{Algorithm}

That initial tree $T$ is a first node of the graph $\graphOfShortestPathTrees(D)$.

\begin{Algorithm}[Traversing $\graphOfShortestPathTrees(D)$]\label{alg:traverse}
  We will maintain a queue, $Q$, of pairs of trees and parameterized distances.
  That queue is initialized with a single shortest-path tree (and pairs of parameterized distances to the target node $t$) obtained from Algorithm~\ref{alg:startnode}.

  While $Q$ is nonempty, pick and remove from $Q$ the next tree $T$, together with the parameterized distances $p_v^T$ from $v$ to $t$.
  For every arc $(v,w)$, compare $p_v^T$ with $d_{vw}+p_w^T$. 
  If they are incomparable add $p_v^T \leq d_{vw}+p_w^T$ to a system of inequalities associated with $T$, and replace the outgoing arc of $v$ by $(v,w)$ to obtain a new tree $T'$.
  Compute the new parameterized distances $p^{T'}$, and check whether $T'$ is a shortest-path tree.
  In that case and if additionally $T'$ has not been considered before, add $T'$ to $Q$.
  Output the triplet of the tree $T$, the distance function $p^T$ and the system of inequalities describing the region of $T$, when there is no arc left to compare.
\end{Algorithm}

\begin{remark}
  Algorithm~\ref{alg:traverse} is a breadth first search on $\graphOfShortestPathTrees(D)$; cf.\ \cite[Chapter~1]{Tarjan:1983}.
  The order in which the traversal is organized is not particularly relevant.
  Similarly, the initial shortest-path tree constructed in Algorithm~\ref{alg:startnode} could be replaced by any other shortest-path tree with a non-empty feasible region of parameters.
\end{remark}

Let us now determine bounds on the number $\combinatorialTypes=\combinatorialTypes(D,t)$ of shortest-path trees with target $t$ and weighted adjacency matrix $D$.
That is also the number of nodes in $\graphOfShortestPathTrees(D)$ and a crucial parameter for the complexity of the Algorithm~\ref{alg:traverse}.
We will use two further definitions in our analysis.
We call a variable \emph{active} in a shortest-path tree when it occurs in the weight of one of its arcs.
Moreover, an adjacency matrix is \emph{generic} if no two collections of arcs yield the same sum of weights, and this applies to a variable weight $x_i$ by taking the minimal value $\lambda_i$ and the maximal value $\mu_i$ into account.
Formally, we consider the set of weights
\[
  W \ := \ \SetOf{ d_{vw} }{ v,w\in[n] ,\, d_{vw}\text{ constant}} \; \cup \; \SetOf{ \lambda_i }{ i\in[k] } \; \cup \; \SetOf{ \mu_i }{ i\in[k],\, \mu_i<\infty }
\]
whose elements we require to be pairwise distinct such that, moreover, any two subsets of $W$ (independent of their cardinalities) yield distinct weight sums.
For instance, genericity entails that the shortest-path tree of any connected subgraph (with fixed target) is unique, if it involves only arcs of constant weight.
This will simplify the proofs, but it does not restrict the applicability of our results; see Remark~\ref{rem:generic} below.

The proof of our main result uses the analysis of a special class of directed graphs, where it matters how the arcs with variable weights are distributed.
For this we have a technical lemma, which we split into two parts, to increase the readability.

\begin{lemma}\label{lem:varnodes:ell=0}
  Let $D\in\TT[x_1,\dots,x_{k}]^{n\times n}$ be the weighted adjacency matrix of a directed graph with $n=2k+1$ nodes with the following arcs and weights:
  Each node $u\in[k]$ has a unique outgoing arc, to the node $u+k$, with variable weight $d_{u,u+k}=x_u$.
  Each node $v\in\{k+1,\dots,2k\}$ has arcs to the nodes $w\in [k]\cup\{n\}$ of arbitrary positive constant weight $d_{v,w}\in\RR_{>0}\cup\{\infty\}$.
  There are no other arcs, and there exists at least one path from each node to $n$ whose weight has a finite constant term.
  If $D$ is generic, then $\combinatorialTypes(D,n)\leq k!$.
\end{lemma}
Note that there are no arcs between any two nodes in $[k]$ nor between any two nodes in $\{k+1, \dots, 2k\}$.
The arcs with variable weights form a matching.
\begin{proof}
Without loss of generality, and since $D$ is generic, we may assume that $k+1$ is the unique node for which the weight $d_{k+1,n}$ is minimal among all weight of arcs ending at~$n$.
Then the shortest path from $1$ to $n$ (via $k+1$) is unique, as $d_{k+1,n} < d_{k+u,n}$ implies $d_{k+1,n} < d_{k+1,u} + x_u +d_{k+u,n}$ for $u\in\{2,\dots,k\}$ and $x_u\geq 0$.
Every shortest-path tree $T$ partitions the set of  nodes into two parts, namely those whose shortest path to $n$ passes through $1$ and the others.
Since $u$ for $u\in[k]$ has a unique outgoing arc, the nodes $u$ and $u+k$ are always in the same part.
In this way, the partition of the nodes induced by $T$ is described by a partition of $[k]$ into two parts.
Let $T_1$ be the subtree of $T$ that includes the node $1$ and all nodes $u$, $u+k$ with a shortest path through $1$; this is a shortest-path tree (on that subset of nodes) directed to $1$.
Similarly, the subtree $T'$ of all the other nodes except for $k+1$ is a shortest-path tree directed to $n$.
Combining the arcs in $T_1$ and $T'$ with the arc $(k+1,n)$ yields the arcs of $T$; as $k+1$ is fixed the trees $T_1$ and $T'$ determine $T$.

For a directed graph that satisfies the assumptions of the current Lemma~\ref{lem:varnodes:ell=0}, let $\sigma_k$ denote the maximal number of shortest-path trees directed to the fixed target node $n$.
Now the two subgraphs on nodes of the partition induced by the tree $T$ are of the same shape as in our assumption.
So there is some index $j$ such that $\sigma_j$ counts the shortest-path trees in $T_1$ (with target node $1$), and $\sigma_{k-j-1}$ counts the shortest-path trees in $T'$ (with target node $n$).
There are $\tbinom{k-1}{j}$ ways to choose a subset of $\{2,\dots,k\}$ of size $j$ and hence we arrive at
\[
  \sigma_k \ \leq \ \sum_{j=0}^{k-1} \binom{k-1}{j} \cdot \sigma_j \cdot \sigma_{k-j-1} \enspace.
\]
We have $\sigma_0 = 0! = \sigma_1 = 1! = 1$.
Finally, induction on $k$ gives $\sigma_k\leq k! = \sum_{j=0}^{k-1} \binom{k-1}{j} j! (k-j-1)!$.
\end{proof}

Setting $\ell=0$ in the next result is precisely Lemma~\ref{lem:varnodes:ell=0}, which we just proved.
Before we enter the precise formulation we consider two examples.

\begin{example}
  The bound of Lemma~\ref{lem:varnodes:ell=0} is tight.
  It is attained, e.g., if the arcs $(u+k,v)$ have weight $2^{u-1}+u-2^{v}-v$ for all $u>v$ with $u,v\in[k]$, the arcs $(u+k,n)$ have weight $2^{u-1}+u-1$ for $u\in[k]$, and there are no other arcs of constant weight.
\end{example}

\begin{example}\label{ex:lemma}
  Figure~\ref{fig:graphlemma} illustrates the conditions of Lemma~\ref{lem:varnodes} below, where $k=3$ and $\ell=2$.
  It shows a directed graph with $9=2\cdot k+\ell+1$ nodes, $5=k+\ell$ variables, and ten arcs of constant weight.
  While the weighted adjacency matrix is not generic, the paths to $t$ have distinct weights.
  The graph in Figure~\ref{fig:graphlemma} has $30 \leq 120 = (k+\ell)! \leq 125 = (k+\ell)^k$ shortest-path trees directed to the node $n=9$.
  If the two nodes $4$ and $5$ are deleted with all their in- and outgoing arcs and nodes are relabeled one gets an example of a graph with parameters $k=3$ and $\ell=0$.
  This new graph has $6=3!$ shortest-path trees directed to $9$. This number agrees with the bound given in Lemma~\ref{lem:varnodes}.
  However, that bound is not tight for large values of $\ell$. 
\end{example}

\begin{figure}[t]\centering
  \begin{tikzpicture}[scale = 0.8,
                    color = {black}]

  \tikzset{->-/.style={decoration={
  markings,
  mark=at position #1 with {\arrow{>}}},postaction={decorate}}}
  \tikzstyle{linestyle} = [->-=0.87,> = stealth, thick, line cap=round, line join=round];
  \tikzstyle{linestyle2} = [->-=0.8,> = stealth, ultra thick, line cap=round, line join=round, black];

  \coordinate (t) at (11.2,0);

  \coordinate (v1) at (3, 2);
  \coordinate (v2) at (3, 0);
  \coordinate (v3) at (3, -2);

  \coordinate (u1) at (9.3, 2);
  \coordinate (u2) at (0,0);

  \coordinate (w1) at (6, 2);
  \coordinate (w2) at (6, 0);
  \coordinate (w3) at (6, -2);

   \draw[linestyle2] (v1) to (w1);
   \node[fill=white,inner sep=1pt] at ($0.55*(v1)+0.45*(w1)$) {$x_1$};
   \draw[linestyle2] (v2) to (w2);
   \node[fill=white,inner sep=1pt] at ($0.55*(v2)+0.45*(w2)$) {$x_2$};
   \draw[linestyle2] (v3) to (w3);
   \node[fill=white,inner sep=1pt] at ($0.55*(v3)+0.45*(w3)$) {$x_3$};

   \draw[linestyle2] (u1) to (t);
   \node[fill=white,inner sep=1pt] at ($0.55*(u1)+0.45*(t)$) {$x_4$};
   \draw[linestyle2] (u2) to (v2);
   \node[fill=white,inner sep=1pt] at ($0.55*(u2)+0.45*(v2)$) {$x_5$};

   \draw[linestyle] [use Hobby shortcut] (w1) .. (9.8, 2.5) .. (t);
   \node[fill=white,inner sep=1pt] at (7.8, 2.9) {$2$};
   \draw[linestyle] (w1) to (u1);
   \node[fill=white,inner sep=1pt] at ($0.5*(w1)+0.5*(u1)$) {$1$};

   \draw[linestyle] (w2) to (t);
   \node[fill=white,inner sep=1pt] at ($0.7*(w2)+0.3*(t)$) {$5$};
   \draw[linestyle] (w2) to (u1);
   \node[fill=white,inner sep=1pt] at ($0.5*(w2)+0.5*(u1)$) {$4$};
   \draw[linestyle] (w2) to (v1);
   \node[fill=white,inner sep=1pt] at ($0.5*(w2)+0.5*(v1)$) {$1$};

   \draw[linestyle] (w3) to[bend right] (t);
   \node[fill=white,inner sep=1pt] at (9,-1.7) {$12$};
   \draw[linestyle] (w3) to[bend right] (u1);
   \node[fill=white,inner sep=1pt] at (7.8,-.95) {$11$};
   \draw[linestyle] [use Hobby shortcut] (w3) .. (2.5,-2.5) .. (u2);
   \node[fill=white,inner sep=1pt] at (1.3, -1.8) {$1$};
   \draw[linestyle] [use Hobby shortcut] (w3) .. (2.2,-2.8) .. (-.6,0) .. (v1);
   \node[fill=white,inner sep=1pt] at (0, 1.5) {$8$};
   \draw[linestyle] (w3) to (v2);
   \node[fill=white,inner sep=1pt] at ($0.5*(w3)+0.5*(v2)$) {$3$};

   \node [circle, ultra thick, draw,fill=white] at (t) {$9$};
   \node [circle, ultra thick, draw,fill=white] at (w1) {$6$};
   \node [circle, ultra thick, draw,fill=white] at (w2) {$7$};
   \node [circle, ultra thick, draw,fill=white] at (w3) {$8$};

   \node [circle, draw, fill=white] at (u1) {$4$};
   \node [circle, draw, fill=white] at (u2) {$5$};

   \node [circle, ultra thick, draw,fill=white] at (v1) {$1$};
   \node [circle, ultra thick, draw,fill=white] at (v2) {$2$};
   \node [circle, ultra thick, draw,fill=white] at (v3) {$3$};

\end{tikzpicture}
  \caption{The directed graph from Example~\ref{ex:lemma}.}
  \label{fig:graphlemma}
\end{figure}

\begin{lemma}\label{lem:varnodes}
  Let $D\in\TT[x_1,\dots,x_{k+\ell}]^{n\times n}$ be the weighted adjacency matrix of a directed graph with $n=2k+\ell+1$ nodes with the following arcs and weights:
  Each node $u\in[k]$ has a unique outgoing arc, to the node $u+k+\ell$, with variable weight $d_{u,u+k+\ell}=x_u$.
  Each node $u\in\{k+1,\dots,k+\ell\}$ has a unique outgoing arc, to some node $w<u$ or $w>k+\ell$, with variable weight $d_{u,w}=x_{u}$.
  Each node $v\in\{k+\ell+1,\dots,2k+\ell\}$ has arcs to the nodes $w\in [k+\ell]\cup\{n\}$ of arbitrary positive constant weight $d_{v,w}\in\RR_{>0}\cup\{\infty\}$.
  There are no other arcs, and there exists at least one path from each node to $n$ whose weight has a finite constant term.
  If $D$ is generic, then $\combinatorialTypes(D,n) \leq \min\{(k+\ell)!, (k+\ell)^k\}$.
\end{lemma}

\begin{proof}
We first prove that $\combinatorialTypes(D,n) \leq(k+\ell)^k$. The successor of each node in $[k+\ell]$ is unique. In a shortest-path tree there is exactly one arc from each of the remaining $k$ nodes (excluding the target node) to one of the nodes in $[k+\ell]\cup\{n\}$ which is not the predecessor. Thus the number of shortest-path trees satisfies $\combinatorialTypes(D,n) \leq (k+\ell)^k$.

Now we show that $\combinatorialTypes(D,n) \leq (k+\ell)!$. Let us relabel the node $n=2k+\ell+1$ by $n=2(k+\ell)+1$, and add $\ell$ new nodes $\{2k+\ell+1,\ldots,2k+2\ell\}$. Furthermore, we replace each arc $(u,v)$ with $u\in\{k+1,\ldots,k+\ell\}$ by the two arcs $(u,u+k+\ell)$ and $(u+k+\ell,v)$ of weight $x_u$ and weight $\epsilon_u$ respectively, for some small positive number $\epsilon_u>0$ such that the resulting graph has a generic adjacency matrix.
This process does not change the number of shortest-path trees. However, this new graph meets the conditions of Lemma~\ref{lem:varnodes:ell=0} with $k+\ell$ variables. Thus $\combinatorialTypes(D,n) \leq (k+\ell)!$.
\end{proof}

We consider the function $\Phi:\NN\times\NN\to\NN$ defined as
\[
  \Phi(n,k) \ := \ \sum_{i = 0}^k \binom{k}{i} (i+1)^{n-i-1} \enspace .
\]
For instance, we have $\Phi(n,0)=1$ and $\Phi(n,1)=1+2^{n-2}$.
Clearly,
\[
  \Phi(n,k) \ \leq \ \sum_{i = 0}^k \binom{k}{i} (k+1)^{n-i-1} \ = \ (k+1)^{n-1} \cdot \Bigl(1+\frac{1}{k+1}\Bigr)^k \ < \ \text{e}\cdot (k+1)^{n-1}\enspace ,
\]
where e is Euler's constant.
\begin{theorem}\label{thm:shortestpathtrees}
  Let $D\in\TT[x_1,\dots,x_k]^{n\times n}$ be the weighted adjacency matrix of a directed graph with $n$ nodes and $m$ arcs.
  Suppose that $D$ is generic and has separated variables (with lower and upper bounds), and let $t\in[n]$ be some node.
  Further, let $\combinatorialTypes=\combinatorialTypes(D,t)$ be the number of shortest-path trees with target node $t$.
  Then
  \[
    \combinatorialTypes \ \leq \ \min\left( \Phi(n,k),\, n^{n-2},\,\binom{m}{n-1} \right) \enspace,
  \]
  and the graph traversal Algorithm~\ref{alg:traverse} computes the shortest-path trees together with an inequality description for each region of $D^*$ in $O(k\cdot m^2\cdot \combinatorialTypes + n^2)$ time.
	The space complexity is bounded by $O(k\cdot (n+m) \cdot \combinatorialTypes)$.
\end{theorem}
\begin{proof}
  Since each spanning tree in a graph with $n$ nodes has only $n-1$ edges, there are at most $\tbinom{m}{n-1}$ (shortest-path) trees in the graph defined by $D$.
  
  Next let us discuss the extremal case $k=n^2-n$.
  Then we have as many variables as possible, say, with weight intervals $[0,\infty]$, and the graph defined by $D$ is $\widetilde{K}_n$, the complete directed graph on $n$ nodes.
  In this case any two arcs and paths are incomparable, and thus all labeled spanning trees of the undirected graph are produced as output.
  By Cayley's formula the complete undirected graph $K_n$ has precisely $n^{n-2}$ labeled spanning trees.
  Note that fixing the target node $t$ in an undirected spanning tree amounts to picking $t$ as the root and directing all edges toward it.
  Since increasing the number of variables cannot decrease the number of shortest-path trees we obtain the second inequality $\combinatorialTypes\leq n^{n-2}$.

  Now we will look into the general case.
  We want to count the number of shortest-path trees toward $t$ with exactly $i$ active variables.
  Fix a set of $i$ variables, which amounts to fixing a set, $A$, of $i$ arcs, as we have separated variables.
  Furthermore, let $\ell$ denote the number of arcs in $A$ whose end point has an outgoing arc in $A$ or is the target node $t$. 
  Now pick some directed spanning tree, $T$, with target $t$ that includes the $i$ arcs in $A$ and no other arc with variable weight.
  Note that such a tree might not exist, whence we may overestimate the number of shortest-path trees.

  In the next step we contract the arcs in $T$ with constant weight whose start point has no incoming arc of variable weight.
  Here we add the weight of a contracted arc $(u,v)$ to the weight of each in-arc of $u$. 
  In this way we keep exactly $i-\ell$ of the arcs with constant weight.
  We arrive at a directed tree with $2i-\ell+1$ nodes and $2i-\ell=(i-\ell)+i$ arcs.
  By construction, this tree is a subgraph of a graph like in Lemma~\ref{lem:varnodes} with parameters $k=i-\ell$ and $\ell$. 
  In fact, the parameter $i-\ell$ overestimates what is called $k$ in Lemma~\ref{lem:varnodes} whenever two arcs with with variable weight point to the same node and its outgoing arc is constant.
  At any rate, there are at most $i^{i-\ell}$ of these trees.

  Undoing the contraction of the constant arcs we obtain shortest-path trees with at most $n-2i+\ell-1$ additional nodes. 
  From any of these additional nodes a shortest path to $t$ either is of constant length or uses an arc with variable weight.
  In the latter case there are $i$ choices for the first arc taken with variable weight while constant shortest paths are unique as $D$ is generic.
  Hence there are at most $(i+1)^{n-2i+\ell-1}$ choices and in total at most $(i+1)^{n-2i+\ell-1}\cdot i^{i-\ell}$ shortest-path trees for a fixed set of $i$ active variables.
  Clearly $i\leq i+1$ and hence $(i+1)^{n-2i+\ell-1}\cdot i^{i-\ell}\leq (i+1)^{n-i-1}$. 
  Since there are $\tbinom{k}{i}$ such sets we conclude that the total number of shortest-path trees satisfies also the final inequality $\combinatorialTypes\leq \Phi(n,k)$.
  
  Now let us estimate the complexity of Algorithm~\ref{alg:traverse} in terms of the number of variables~$k$, the number of nodes~$n$, the number of arcs~$m$ and the number of shortest-path trees~$\combinatorialTypes$.

  The initial step is to compute a shortest-path tree, $T$, with Algorithm~\ref{alg:startnode}.
  This means, first, to create an $n\times n$ adjacency matrix with constant weights, second, to apply Dijkstra's algorithm, and, third, to find the inequality description of the feasible region for $T$.
  That takes $O(n^2)$ for the first two steps and $O(k\cdot n)$ for the third, adding up to $O(k\cdot n+n^2)$.

  The queue $Q$ of the Algorithm~\ref{alg:traverse} treats every shortest-path tree at most once. 
  It follows from Observation~\ref{obs:number_ineq} that such a tree $T$ has at most $m$ arcs that lead to an inequality of the region of $T$, and hence at most $m$ potential neighbors in $\graphOfShortestPathTrees(D)$.
  It takes $O(k\cdot n)$ to update the distances $p^{T'}$ and $O(k\cdot m)$ to check whether $T'$ is a shortest-path tree.
  Thus in total the time complexity of the Algorithm~\ref{alg:traverse} is at most $O(k\cdot n+n^2+k\cdot m\cdot \combinatorialTypes\cdot (n+m))=O(n^2+k\cdot m^2\cdot\combinatorialTypes)$, as $n-1\leq m$.

  In terms of space complexity the Algorithm~\ref{alg:traverse} is a typical breadth-first search: the dominating contribution comes from the queue $Q$, whose length is bounded by $\combinatorialTypes$.
  Each entry of the queue contains a tree, on $n$ nodes, the distances to the target, and the list of constraints of the corresponding region.
  The distance function has $n$ entries, and the at most $m$ edges give the linear constraint in the $k$ indeterminates.
  We infer that the total space complexity is of order at most $O(k\cdot (n+m) \cdot \combinatorialTypes)$, in the unit cost model.
\end{proof}

\begin{figure}[t]\centering
  \begin{tikzpicture}[scale = 0.8,
                    color = {black}]

  \tikzset{->-/.style={decoration={
        markings,
        mark=at position #1 with {\arrow{>}}},postaction={decorate}}}
  \tikzstyle{linestyle} = [->-=0.87,> = stealth, thick, line cap=round, line join=round, dashed];
  \tikzstyle{tree} = [->-=0.87,> = stealth, thick, line cap=round, line join=round];
  \tikzstyle{active} = [->-=0.87,> = stealth, ultra thick, line cap=round, line join=round];
  \tikzstyle{inactive} = [->-=0.87,> = stealth, ultra thick, line cap=round, line join=round, dashed];

  \coordinate (t) at (10,0);

  \coordinate (v1) at (0,2);
  \coordinate (v2) at (0,-2);
  \coordinate (v3) at (4,2);
  \coordinate (v4) at (4,-2);
  \coordinate (v5) at (8,2);
  \coordinate (v6) at (8,-2);

  \draw[tree] (v1) to (v3);
  \draw[linestyle] (v1) to (v4);
  \draw[linestyle] (v2) to (v3);
  \draw[active] (v2) to (v4);
  \draw[active] (v3) to (v5);
  \draw[linestyle] (v3) to (v6);
  \draw[inactive] (v4) to (v5);
  \draw[tree] (v4) to (v6);
  \draw[active] (v5) to (t);
  \draw[tree] (v6) to (t);

  \foreach \i in {1,...,6}{
    \node [circle, ultra thick, draw,fill=white] at (v\i) {$\i$};
  }
  \node [circle, ultra thick, draw,fill=white] at (t) {$t$};

\end{tikzpicture}

% Local Variables: 
% mode: latex
% mode: TeX-PDF
% TeX-master: "../opt-main"
% mode: reftex
% mode: font-lock
% fill-column: 100
% buffer-file-coding-system:utf-8-unix
% End: 
  \caption{The directed graph from Example~\ref{ex:theorem}. Variable arcs are bold, non-tree arcs are dashed.}
  \label{fig:graphtheorem}
\end{figure}

\begin{corollary}
  Let $D$ be a generic weighted adjacency matrix with separated variables.
  Algorithm~\ref{alg:traverse} enumerates all shortest-path trees of $D$ to a fixed node, together with their distance functions and inequality description of their polyhedral cell, that are path-connected in $\graphOfShortestPathTrees(D)$ to a shortest-path tree of a feasible region. In particular, it enumerates a shortest-path tree for every feasible region.  
\end{corollary}
\begin{proof}
Every node of $\graphOfShortestPathTrees(D)$ is a shortest-path tree.
Thus, by definition of an edge in $\graphOfShortestPathTrees(D)$ every neighbor that is unseen will be added to the queue of Algorithm~\ref{alg:traverse} and hence be visited and enumerated.
The shortest-path tree computed in Algorithm~\ref{alg:startnode} corresponds to a region, and by Remark~\ref{rem:dualgraph}, every shortest-path tree of a region is in the same connected component.
\end{proof}

\begin{example}\label{ex:theorem}
  To illustrate the generic case in the proof of Theorem~\ref{thm:shortestpathtrees} we consider the directed graph with seven nodes shown in Figure~\ref{fig:graphtheorem}, where we omit the weights.
  We use the same notation as in the proof.
  The arcs $(1,3)$, $(2,4)$, $(3,5)$, $(4,6)$, $(5,t)$ and $(6,t)$ form the directed spanning tree~$T$.
  The set of arcs with active variables is $A=\{(2,4),(3,5),(5,t)\}$, i.e., $i=3$.
  The only arc with a nonactive variable is $(4,5)$.
  The constant arcs $(1,3)$ and $(6,t)$ are contracted, and $\ell=2$.
  The weight of $(6,t)$ is added to both, $(3,6)$ and $(4,6)$ which now point to the target node; the weight of $(1,3)$ is irrelevant, since the in-degree of $1$ is zero.
\end{example}

\begin{remark}\label{rem:generic}
  The assumption concerning the genericity of the arc weights in Theorem~\ref{thm:shortestpathtrees} is not essential.
  In fact, via picking a fixed total ordering of the arcs one can break ties between two candidates of shortest paths by comparing them lexicographically.
  This technique is known as \enquote{symbolic perturbation} and implemented in our \polymake code; see also Section~\ref{sec:computations}.
\end{remark}

\begin{remark}
  Theorem~\ref{thm:shortestpathtrees} could also be proved without Lemma~\ref{lem:varnodes}.
  However, the current proof additionally shows that the bound given is not tight whenever $\ell>0$.
  Moreover, that argument also allows to estimate the number of feasible regions in Proposition~\ref{prop:feasible}.
\end{remark}

\begin{remark}
  Two variables cannot simultaneously be active if they share the same initial vertex.
  This situation occurs, e.g., when $k\geq n$; hence $\Phi(k,n)$ overestimates the number of shortest-path trees in that range.
  The function $\Phi(k,n)$ is the sum over the maximal number of shortest-path trees with $i\leq k$ active variables.
  These maxima cannot be attained simultaneously if $k\geq 2$; thus $\Phi(k,n)$ overestimates the number of shortest-path trees also for $k\geq 2$.
  Clearly, our bound is tight for (directed) trees, the complete directed graph $\widetilde K_n$ on $n$ nodes with $k=n^2-n$ variables, and on graphs with $k\leq 1$ variables.   
\end{remark}

In general, many trees enumerated by Algorithm~\ref{alg:traverse} and counted in Theorem~\ref{thm:shortestpathtrees} will correspond to regions which are infeasible.
So it is desirable to independently bound the number of feasible regions.
\begin{proposition}\label{prop:feasible}
  Let $D$ be the weighted adjacency matrix of a directed graph on $n$ nodes.
  Then the number of feasible regions, which are induced by the shortest paths to a fixed target node, is at most
  \begin{equation}\label{eq:feasible}
    \sum_{i=0}^k \binom{k}{i} \frac{(n-1)!}{(n-i-1)!} \ \leq  \ \sum_{i=0}^k \binom{k}{i} (n-1)^i \ = \ n^k \enspace .
  \end{equation}
\end{proposition}
\begin{proof}
  There is a shortest-path tree for every feasible region.
  Thus we may estimate the number of feasible regions by counting shortest-path trees.
  Again we do this by fixing a set of $i$ active variables and take only those trees into account that include the corresponding arcs and no other arcs of variable weight.
  Now there are two types of tropical hypersurfaces arising from such a tree.
  Those which correspond to nodes whose outgoing arc weight is one of the active variables, there are $i$ of those, and those nodes whose outgoing arcs are of constant weight of which exist $n-i-1$.
  The common refinement of the first once has at most $i!$ many regions as this is the bound of shortest-path trees in Lemma~\ref{lem:varnodes}.
  The latter hypersurfaces result from nodes which have at most $i+1$ incomparable shortest paths to the target.
  Moreover, the total weights of those paths are linear, up to substitution of variables.
  Hence, the common refinement of these hypersurfaces has at most $\tbinom{n-1}{i}$ regions; see also Remark~\ref{rem:infeasible}.
  There are $\tbinom{k}{i}$ ways to choose $i$ active variables hence the total number of feasible regions is at most
  \[
    \sum_{i=0}^k \binom{k}{i} \cdot i! \cdot \binom{n-1}{i} \ = \ \sum_{i=0}^k \binom{k}{i} \frac{(n-1)!}{(n-i-1)!}  \enspace .
  \]
\end{proof}

The number of feasible regions in Proposition~\ref{prop:feasible} is much smaller than the bound of the number of trees in Theorem~\ref{thm:shortestpathtrees}.
This raises the question why we let Algorithm~\ref{alg:traverse} take such a \enquote{detour}.
For instance, we could employ the parametric Floyd--Warshall algorithm, which runs in $O(k\cdot 2^k n^3)$ time and thus is faster than Algorithm~\ref{alg:traverse}.
However, from Floyd--Warshall we only obtain the Kleene star, an $n{\times}n$ matrix with entries in $\TT[x_1,\ldots,x_k]$, from which it is rather expensive to extract all shortest-path trees or regions. 
Conceptually we could compute the decomposition into regions by a dual convex hull computation in $\RR^{k+1}$ but this requires at least as many as \eqref{eq:feasible} linear constraints.
In Section~\ref{sec:computations} we will see that Algorithm~\ref{alg:traverse} can easily compute instances for $k=10$ and $n$ in the range of hundreds or even thousands.
Convex hull computations of that size are entirely out of question; cf.\ \cite{polymake:2017} for a recent survey on the subject.
Thus a much superior strategy is to first enumerate the trees and then to filter for the feasible regions by linear programming.
As an additional advantage the latter step can be parallelized trivially.

Yet, from a practical point of view, it may also be useful to combine the parameterized Floyd-Warshall algorithm with Algorithm~\ref{alg:traverse}.
The two methods are quite different and thus exhibit different advantages.
Imagine a car driver with a very long drive to her destination city.
The traffic situation in the city depends on the time when she reaches the city.
Therefore, the driver's navigation system might evaluate the parameterized Kleene-star with up-to-date data whenever there is a branching point.
This would lead to a complete choice of reasonable routes at any given time.
Other scenarios are conceivable, where the Kleene star is used to compute a description by inequalities of a feasible region in $O(k\cdot n^2)$ time from a generic feasible point.
Such information could be valuable for network operators or providers.
While investigations of this kind look promising, in the next section we restrict our attention to several instances of one scenario.

\section{Computations}
\label{sec:computations}
\noindent
We report on extensive computational experiments with \polymake \cite{DMV:polymake}.

\subsection{Implementation}
In the following we collect some details on implementing Algorithm~\ref{alg:traverse}. 
The genericity of the matrix $D$ can be achieved by a symbolic perturbation of the arc weights.
Choosing an ordering on all arcs induces a (lexicographic) ordering on arbitrary sets of edges.
In particular, this gives a total ordering on the set of all (shortest-path) trees.
We may pick the ordering on the arcs in such a way that the shortest-path tree produced by Algorithm~\ref{alg:startnode} is minimal.
The lexicographic ordering on the shortest-path trees allows to traverse $\graphOfShortestPathTrees(D)$ without a lookup table or cache.
This can be interpreted in terms of Dijkstra variants based on \enquote{labeling} and \enquote{scanning}; cf.\ \cite[\S7.1]{Tarjan:1983}.
See also, e.g., \cite{Gawrilow:2008} for a dynamic routing algorithm employing that idea.

For maximal speed it is relevant to organize the trees and especially the queue of trees to be processed by means of suitable data structures.
Most importantly, there is an improved version of Dijkstra's algorithm by Fredman and Tarjan \cite{Fredman+Tarjan:1987} based on \emph{Fibonacci-Heaps}.
The latter leads to a complexity of $O(m+n \log(n))$ in the unparameterized setting; see also \cite[\S7.4]{Schrijver03:CO_A}.

An optimized variant of Algorithm~\ref{alg:traverse} has been implemented by Ewgenij Gawrilow in the \polymake software system \cite{DMV:polymake}.
This implementation uses dynamic programming and backtracking to traverse the graph implicitly.
The code is available as an extension to \polymake, version 4.1; see \url{https://polymake.org/extensions/polytropes}.

\subsection{Real-world traffic data}
\label{subsec:TNTP}
To show that our methods are feasible in practice we tried the parameterized Dijkstra Algorithm \ref{alg:traverse} on real-world data sets from the Transportation Networks repository \cite{TransportationNetworks}.
Before we explain our experimental setup we wish to spend a few words on those traffic data.
We focus on the files with the extension \texttt{tntp}.
Each file encodes a directed graph which comes from a road network and additional information about the travel time along the arcs.
For every arc $(u,v)$ the \emph{link travel time}, depending on the flow $x$, is the quantity
\begin{equation}\label{eq:BPR}
  \realtravel_{(u,v)}(x) \ = \ \freetravel_{(u,v)} \cdot \bigl(1+\bias(\tfrac{x}{\capacity_{(u,v)}})^{\power}\bigr) \enspace ,
\end{equation}
where $\freetravel_{(u,v)}$ is the \emph{free flow travel time}, $\bias$ is the \emph{bias}, $\capacity_{(u,v)}$ is the capacity and $\power$ is the \emph{power}.
This formula was devised by the \emph{Bureau of Public Roads (BPR)}, a predecessor organization of the \emph{Federal Highway Administration (FHWA)} of the United States. 
The \texttt{tntp}-files contain all these parameters for every arc.
In our data we found $\bias=1$ and $\power=4$ throughout; these parameters are used to model certain nonlinearities extracted from empirical data.
Usually there are also some \emph{zones}, i.e., nodes which no traffic can go through.
For a more comprehensive discussion of the data and the parameters we recommend the web site \cite{TransportationNetworks} and the references given there.
\begin{figure}[t]\centering
  \input{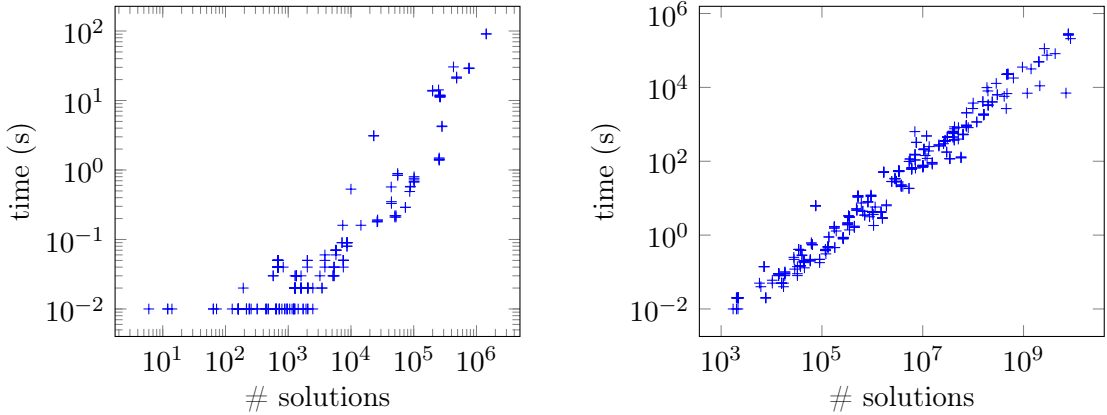}
  \caption{Data set \enquote{Berlin-Mitte-Center}.  \polymake running times versus number of solutions, both log-scaled.
    Left: $p=0.05$ yielding 25 variable weights.  Right: $p=0.08$ yielding 42 variable weights (computation for one node aborted after a week).
}
\label{fig:tntp-timings}
\end{figure}

All timings in our experiments were taken on Intel(R) Core(TM) i7-3930K CPU @ 3.20GHz, turbo 3.73 GHz (BogoMIPS: 6399.75) with openSUSE 42.3 (Linux 4.4.132-53).
The memory consumption did not exceed 200~MB.

\subsection{One graph, all nodes}\label{subsubsec:berlin-mitte-center}
In the first scenario we considered the \enquote{Berlin-Mitte-Center} data set from \cite{TransportationNetworks}, which was originally provided by Jahn et al.\ \cite{Moehring:2005}.
This network describes a directed graph with 398 nodes and 871 arcs.
The first 36 nodes are zones.
Since no traffic can go through a zone, we removed them, along with the incident arcs.
The remaining network has 362 nodes and 583 arcs.

From this data file we created random instances in the following way.
As an additional parameter we fix some probability $p\geq 0$.
Each arc independently receives a variable weight with probability~$p$.
For the constant weights we take the free flow travel times, which are always positive.
Each variable weight is constrained to an interval from the free flow travel time to the link travel time \eqref{eq:BPR} for a flow value set to a random proportion of the link capacity.
That is, e.g., for $p=0$ we get a usual weighted digraph where the arc weights are the free flow travel times.
For positive $p$ we get some variable weights which are intervals $[\freetravel_{(u,v)}, \realtravel_{(u,v)}(r\cdot\capacity_{(u,v)})]$ with $0\leq r\leq 1$, and $0<r<1$ almost surely.
This is the scenario discussed in Section~\ref{sec:dijkstra}.
The complexity of Algorithm \ref{alg:traverse} is primarily controlled by the number of arcs with variable arc weights.
Moreover, for a fixed graph that complexity is proportional to the size of the output, i.e., the number of combinatorial types of shortest-path trees.
So, in order to obtain a computationally feasible setup, the probability $p$ cannot be too high.
That is, on most of our arcs the flow is set to zero (and the arc weight is $\freetravel_{(u,v)}$), while on a small percentage of the arcs the flow is between zero and some fraction of the capacity (and the arc weight is a variable with lower bound $\freetravel_{(u,v)}$).
In this way our experiment models the situation early in the morning, when most roads are still empty and the first few vehicles start to enter the traffic.

For the first experiment, by setting the probability to $p=0.05$, we obtained 25 arcs with variable weights, and this is about 4.3\% of the total number of arcs.
The second experiment is similar, with $p=0.08$ and 42 variable weights (about 7\% of the arcs).
For both instances we applied the parameterized Dijkstra algorithm to all the 362 nodes.
Figure~\ref{fig:tntp-timings} has an overview of the timings.

For $p=0.05$ most computations could be completed by \polymake within less than a second.
The largest one took nearly 100 seconds with more than one million combinatorial types of shortest-path trees.
This network is displayed in Figure~\ref{fig:bmc+p005+v25+nz}.

\begin{figure}\centering
  \includegraphics[height=.67\textheight]{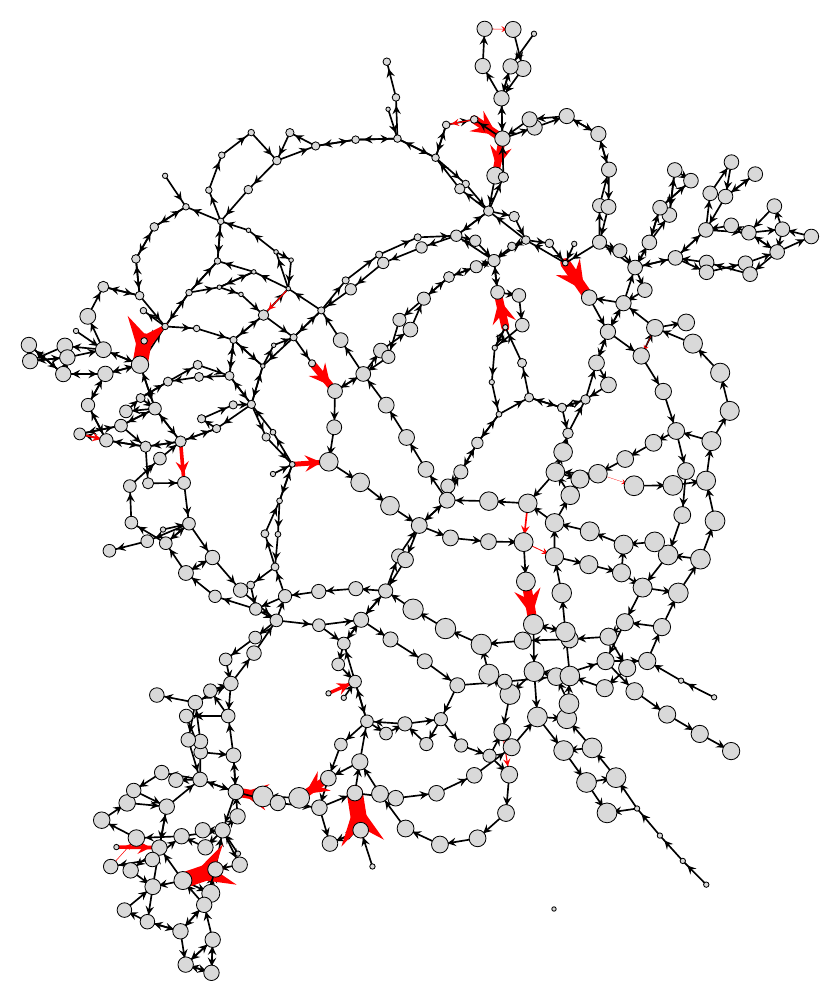} % standalone version of the above
  \caption{Data set \enquote{Berlin-Mitte-Center} with $p=0.05$ resulting in 25 variable weights.
    Arcs with variable weights are red; their line width is proportional to the difference between maximum and minimum travel times.
    Node sizes proportional to $\log\log(\#\text{solutions})$.
    Layout obtained via \neato from the \graphviz package \cite{graphviz}.}
  \label{fig:bmc+p005+v25+nz}
\end{figure}

\begin{table}[!ht]
  \centering
  \caption{The maximum and average performance on the the data from \cite{TransportationNetworks}.  Ten runs each with $k=10$ variables.  The bar--whisker plots show logarithmic performance.}
  \label{tab:tntp-all}
  \renewcommand{\arraystretch}{0.9}
	\begin{tabular*}{0.95\linewidth}{l@{\hspace{0.3cm}}rr@{\hspace{1cm}}rrr@{\hspace{0.5cm}}rrr}
		\toprule
		& \multirow{2}{*}{$n$} & \multirow{2}{*}{$m$} & \multicolumn{2}{c}{Maximum} & \multicolumn{2}{c}{Average}\\
		&&&\# Sol.& Sol./sec &\# Sol.&Sol./sec\\
		\midrule
		Anaheim & 378 & 796 		&    32 &  35199.5 &    5.4 &  6189.77\\
		Barcelona & 910 & 1957 		&  5184 &  12548.4 &  524.7 &  1997.03\\
		Berlin-Center & 12116 & 19724 	& 36480 & 417772.0 & 4622.4 & 68506.98\\
		Berlin-Mitte-Center & 362 & 583 &   144 &  55907.3 &   19.9 & 11500.81\\	
		Berlin-MPFC & 877 & 1410	&  4590 &  59761.6 &  524.4 & 23118.37\\	
		Berlin-Pberg-Center & 314 & 451	&     4 & 113636.0 &    1.6 & 40322.51\\ 	
		Berlin-Tiergarten & 335 & 560 	&     8 &  14670.6 &    4.5 &  7865.83\\
		ChicagoRegional & 11192 & 35436 &   192 &   5094.8 &   36.8 &   981.14\\
		ChicagoSketch & 546 & 2176 	&   512 & 202605.0 &  103.7 & 41586.64\\	
		Berlin-Fhain-Center & 201 & 339 &    90 & 120967.0 &   12.7 & 42332.21\\
		Hessen-Asym & 4415 & 6184 	&     1 &    172.6 &    1.0 &   167.29\\
		Philadelphia & 11864 & 30779 	&    24 &    713.0 &   11.2 &   328.49\\
		Sydney & 29849 & 67381 		&    12 &    150.6 &    9.2 &   114.05\\
		Terrassa-Asym & 1554 & 2953 	&    32 &  11790.8 &    4.1 &  1522.54\\
		Winnipeg-Asym & 903 & 1923 	&     4 &   2423.2 &    1.6 &   978.36\\
		Winnipeg & 905 & 2284 		&   160 &   3530.2 &   18.8 &  1436.35\\
	%berlin-mitte-prenzlauerberg-friedrichshain-center 
   \bottomrule
   \end{tabular*}
   \input{logboxdiagram.tex}
\end{table}

The case where $p=0.08$ is quite different.
For some nodes the computations took several hours, and one computation was aborted after more than one week.
By and large this shows the limits of our approach.
Note that not only the total number of variable arc weights matter but also how clustered they are near the target node; this can also be seen in Figure~\ref{fig:bmc+p005+v25+nz} in the smaller case $p=0.05$.
The largest complete computations produced several billions of shortest-path trees.
The diagrams in Figure~\ref{fig:tntp-timings}, which are log-scaled in both directions, reflect the output-sensitivity of Algorithm~\ref{alg:traverse} as predicted by Theorem~\ref{thm:shortestpathtrees}.

\subsection{Many graphs, some nodes}
In the second scenario we looked at all the data sets (\texttt{\_{}net.tntp} files) from \cite{TransportationNetworks}.
These were preprocessed as in the first scenario, i.e., by removing the zones.
This lead to excluding the three smallest data sets Austin, Braess, and SiouxFalls because too few arcs remained.
As a result we processed 16 directed graphs.
The most demanding one is Berlin-Center with $n=12116$ nodes and $m=19724$ arcs.

This time we fixed the number of variable arcs a priori to $k=10$, and this set of arcs was chosen uniformly at random.
On each variable arc $(u,v)$ we took the interval $[\freetravel_{(u,v)},2\cdot\freetravel_{(u,v)}]$.
We picked a node uniformly at random as the root, which the shortest-path trees are directed to; and this experiment was repeated ten times per instance.

Qualitatively the parameterized Dijkstra Algorithm \ref{alg:traverse} behaves exactly as in the first scenario in Section~\ref{subsubsec:berlin-mitte-center}.
The running times vary considerably, but the predominant factor is the total number of solutions.
This is consistent with our theoretical analysis of the running time from Theorem~\ref{thm:shortestpathtrees}.
And this also agrees with what we observed experimentally in Figure~\ref{fig:tntp-timings} for the first scenario.
Instead of the timings itself, Table~\ref{tab:tntp-all} gives basic statistical information about the \emph{performance}, which we define as the number of solutions per second.
Since we also list the (maximum and the average of) the number of solutions the actual running times can be deduced if necessary.
Here we have fewer variables (but several much larger graphs), and thus the fluctuations are larger.
Again this is no surprise; compare the left and the right diagram in Figure~\ref{fig:tntp-timings}.
A more detailed idea about the entire statistics can be derived from the bar--whisker plots below Table~\ref{tab:tntp-all}.
For the decadic logarithm of the performance it shows the minimum, the 25\% percentile, the median, the 75\% percentile and the maximum per data set.

We think that even the fairly small number of ten random samples per graph suffices to show that the overall behavior of Algorithm \ref{alg:traverse} and its \polymake implementation is well captured by the comprehensive analysis in Section~\ref{subsubsec:berlin-mitte-center}.

\section{Concluding remarks} % and open questions}
\label{sec:concluding}
\noindent
In Section~\ref{sec:computations} we provided experimental evidence to show that our approach is viable in practice, provided that the output size is moderate.
Indeed, we are not aware of any other limiting factor.
For instance, other models of randomly picking variable link travel times significantly change the running times only as far as the total number of shortest-path trees is affected.
This is supported by Table~\ref{tab:tntp-all} which exhibits that the running time per solution found, on a logarithmic scale, stays in a narrow range over a wide selection of rather different networks.
That is to say, the diagram to the right of Figure~\ref{fig:tntp-timings} captures our algorithm's asymptotic behavior on sufficiently large networks rather well.

Our assumption on separated variables models parametric shortest-path problems with the maximal degree of independence among the parameters.
It is conceivable that our approach can be extended to more elaborate settings, but at the price of a greater technical overhead in the analysis.

\bibliographystyle{alpha}
\bibliography{References}

\end{document}